\documentclass[11pt]{amsart}

\usepackage{enumerate, amsmath, amsthm, amsfonts, amssymb, xy,  mathrsfs, graphicx, paralist, lscape}
\usepackage[usenames, dvipsnames]{color}
\usepackage[margin=1in]{geometry} 
\usepackage{tikz}
\usetikzlibrary{matrix,arrows}

\numberwithin{equation}{section}
\newtheorem{theorem}[equation]{Theorem}

\newtheorem{proposition}[equation]{Proposition}
\newtheorem{lemma}[equation]{Lemma}
\newtheorem{corollary}[equation]{Corollary}

\theoremstyle{definition}
\newtheorem{rmk}[equation]{Remark}
\newenvironment{remark}[1][]{\begin{rmk}[#1] \pushQED{\qed}}{\popQED \end{rmk}}
\newtheorem{rmks}[equation]{Remarks}

\newtheorem{eg}[equation]{Example}
\newenvironment{example}[1][]{\begin{eg}[#1] \pushQED{\qed}}{\popQED \end{eg}}
\newtheorem{defn}[equation]{Definition}
\newenvironment{definition}[1][]{\begin{defn}[#1]\pushQED{\qed}}{\popQED \end{defn}}
\newtheorem{ques}[equation]{Question}

\newtheorem{notn}[equation]{Notation}

\usepackage[in]{fullpage}
\usepackage[colorlinks=true, pdfstartview=FitV, linkcolor=blue, citecolor=blue, urlcolor=blue]{hyperref}

\newcommand{\cC}{\mathcal{C}}

\newcommand{\cD}{\mathcal{D}}

\newcommand{\cO}{\mathcal{O}}

\newcommand{\bc}{\mathbf{c}}

\newcommand{\bd}{\mathbf{d}}

\newcommand{\br}{\mathbf{r}}

\newcommand{\zO}{\ensuremath{\Omega}}
\newcommand{\za}{\ensuremath{\alpha}}
\newcommand{\zb}{\ensuremath{\beta}}

\newcommand{\zd}{\ensuremath{\delta}}
\newcommand{\zg}{\ensuremath{\gamma}}

\newcommand{\ZZ}{\mathbb{Z}}
\newcommand{\CC}{\mathbb{C}}

\renewcommand{\phi}{\varphi}
\renewcommand{\emptyset}{\varnothing}

\renewcommand{\tilde}[1]{\widetilde{#1}}

\newcommand{\comment}[1]{}
\newcommand{\setst}[2]{\left\{#1\, \mid \, #2 \right\}}

\makeatletter
\def\Ddots{\mathinner{\mkern1mu\raise\p@
\vbox{\kern7\p@\hbox{.}}\mkern2mu
\raise4\p@\hbox{.}\mkern2mu\raise7\p@\hbox{.}\mkern1mu}}
\makeatother

\DeclareMathOperator{\im}{im} 

\DeclareMathOperator{\coker}{coker}
\newcommand{\Hom}{\operatorname{Hom}}

\DeclareMathOperator{\rank}{rank}

\DeclareMathOperator{\rep}{rep}
\DeclareMathOperator{\Mat}{Mat}

\newcommand{\GL}{\mathbf{GL}}

\newcommand{\Gr}{\mathbf{Gr}}

\newcommand{\onto}{\twoheadrightarrow}
\newcommand{\into}{\hookrightarrow}
\newcommand{\xto}[1]{\xrightarrow{#1}}

\usepackage{blkarray}
\usepackage{mathtools}
\usepackage{ stmaryrd }
\usetikzlibrary{fit,matrix,arrows,backgrounds,shapes.misc,shapes.geometric,patterns,calc,positioning}
\usepackage{ifthen}
\usepackage[normalem]{ulem}
\usepackage{tikz-cd}

\renewcommand{\rep}{\mathtt{rep}}

\newcommand{\red}[1]{\textcolor{red}{#1}}

\begin{document}

\title{Type $D$ quiver representation varieties,\\double Grassmannians, and symmetric varieties}
\author{Ryan Kinser}
\address{Ryan Kinser \\ Department of Mathematics\\ University of Iowa \\ Iowa City, IA, USA }
\email{ryan-kinser@uiowa.edu}
\thanks{This work was supported by a grant from the Simons Foundation (636534, RK)}

\author{Jenna Rajchgot}
\address{Jenna Rajchgot\\ Department of Mathematics and Statistics\\ McMaster University\\ Hamilton, ON, Canada}
\email{rajchgoj@mcmaster.ca}
\thanks{J.R. is partially supported by NSERC Grant RGPIN-2017-05732.}

\begin{abstract}
We unify aspects of the equivariant geometry of type $D$ quiver representation varieties, double Grassmannians, and symmetric varieties $GL(a+b)/GL(a)\times GL(b)$; in particular we translate results about singularities of orbit closures, combinatorics of orbit closure containment, and torus equivariant $K$-theory between these three families.
These results are all obtained from our generalization of a construction of Zelevinsky for type $A$ quivers to the type $D$ setting.
More precisely, we give explicit embeddings with nice properties of homogeneous fiber bundles over type $D$ quiver representation varieties into these symmetric varieties.
\end{abstract}



\maketitle

\tableofcontents

\section{Introduction}
\subsection{Context}\label{sect:context}
This paper continues our program of unifying problems about the equivariant geometry of representation varieties of Dynkin quivers with the corresponding problems for Schubert varieties in multiple flag varieties.  The problems we consider are:
\begin{enumerate}
\item characterizing singularities of orbit closures;
\item combinatorial description of the poset of orbit closures with respect to containment (generalized Bruhat order);
\item formulas for classes of orbit closures in equivariant $K$-theory.
\end{enumerate}
Solutions to these problems for orbit closures in type $A$ quiver representation varieties are governed by solutions to the corresponding problems for Schubert varieties in type $A$ flag varieties (see Section \ref{sect:literature}), and vice versa (see Remark \ref{SchubToTypeA}). 
In the present work, we establish an analogous result for type $D$ quiver representation varieties. We unify the above three problems for orbit closures in type $D$ quiver representation varieties, double Grassmannians (which are particular multiple flag varieties), and certain symmetric quotients of general linear groups.

\medskip

Representation varieties of quivers naturally arise throughout various areas of mathematics, including:
\begin{itemize}
\item commutative algebra, as generalizations of determinantal varieties;
\item Lie theory, through Lusztig's geometric construction of canonical bases;
\item algebraic geometry, through the theory of degeneracy loci of vector bundle maps;
\item representation theory of finite-dimensional associative algebras, as varieties parametrizing framed representations.
\end{itemize}
A more detailed list of references for these can be found in the introduction of the article \cite{kinserICRA}.  A literature comparison specific to this paper is found in Section \ref{sect:literature} below.

On the other hand, symmetric varieties \cite{Springer85} and the more general spherical varieties \cite{Brion95,Brion12,Perrin14} provide a richly developed geometry for us to draw on, due to their long history going back to the classical theory of symmetric spaces \cite{Cartan26,Cartan27}.  
These varieties, particularly spherical multiple flag varieties \cite{MWZ,AP16}, share many features in common with  classical flag varieties.

In this paper, we work with double Grassmannians $X=\Gr_a(k^n) \times \Gr_b(k^n)$, equipped with the diagonal right action of $G=GL(n)$.  In the special case $n=a+b$, we can take the block diagonal subgroup $K=GL(a) \times GL(b) \leq G$, and consider the quotient (known as a symmetric variety\footnote{This terminology often assumes working over a field of characteristic not 2, but all our results hold in arbitrary characteristic.}) $K \backslash G$ with respect to the left action of $K$ on $G$.
It carries a natural right action of $G$ by multiplication, and $X$ has an open $G$-orbit which is $G$-equivariantly isomorphic to $K \backslash G$.
As in the classical theory of flag varieties, we are interested in orbit closures with respect to the actions of a Borel subgroup $B \leq G$ on $K\backslash G$ and $X$.  The $B$-actions are inherited by restricting the $G$-actions.  

For a quiver $Q$ and dimension vector $\bd$, we denote by $\rep_Q(\bd)$ and  $GL(\bd)$ the corresponding representation variety and base change group, respectively (see Section \ref{sect:quivers}).
It is not difficult to see that problems (1), (2), and (3) above for $B$-orbit closures in $X$ (thus also in $K\backslash G$) can be realized as special cases of the same problems for $GL(\bd)$-orbit closures in type $D$ quiver representation varieties.
We conceptually summarize this with the following chain, indicating that solutions to problems (1), (2), (3) for the class of varieties at the target of an arrow govern the corresponding solutions for the class of varieties at the source.
\[
\boxed{\substack{B\text{-orbit closures}\\ \text{in } K\backslash G}} \hookrightarrow 
\boxed{\substack{B\text{-orbit closures}\\ \text{in }\Gr_a(k^n)\times \Gr_b(k^n)}} \hookrightarrow
\boxed{\substack{GL(\bd)\text{-orbit closures}\\ \text{in }\rep_Q(\bd),\ Q \text{ type } D}} 
\]
We give a short but complete explanation in Section \ref{sect:GrassinD} for the reader's convenience.

In this paper, we complete the story by realizing problems (1), (2), and (3) above for arbitrary type $D$ quiver representation varieties as special cases of the same problems for $B$-orbit closures in certain symmetric varieties $K\backslash G$ (Theorem \ref{thm:mainTheorem}).  
\[
\boxed{\substack{B\text{-orbit closures}\\ \text{in } K\backslash G}} \hookrightarrow 
\boxed{\substack{B\text{-orbit closures}\\ \text{in } \Gr_a(k^n)\times \Gr_b(k^n)}} \hookrightarrow
\boxed{\substack{GL(\bd)\text{-orbit closures}\\ \text{in }\rep_Q(\bd),\ Q \text{ type } D}} \hookrightarrow
\boxed{\substack{B'\text{-orbit closures}\\ \text{in } K'\backslash G'}} 
\]
In contrast with the first two links, establishing the third is not at all straightforward.  Our proof relies on explicit embeddings, with nice properties, of type $D$ quiver representation varieties inside symmetric quotient varieties.  
Discovering these embeddings and establishing their properties uniformly for arbitrary orientation was the most significant challenge of this work, resolved by enlarging the quivers under consideration to work outside of Dynkin type.

\subsection{Summary of results and methods}
The main results of this paper can be summarized as follows, retaining the notation of the previous subsection.  References and remarks about equivariant $K$-theory are found in Section \ref{sect:Ktheory}.

\begin{theorem}\label{thm:mainTheorem}
Given a type $D$ quiver $Q$ with dimension vector $\bd$, 
there is an associated $G=GL(a+b)$ and subgroup $K=GL(a)\times GL(b)$ embedded diagonally in $G$, open subvariety $U\subseteq K\backslash G$, and parabolic subgroup $P\subseteq G$, such that the following hold.
\begin{enumerate}[(i)]
\item There is an injective, order preserving map of partially ordered sets
\begin{equation}\label{eq:reptoKG}
\begin{split}
\left\{\begin{tabular}{c} $GL(\bd)$-orbit closures\\ in $\rep_Q(\bd)$\end{tabular} \right\}
& \to
\left\{\begin{tabular}{c} $P$-orbit closures\\ in $K\backslash G$ \end{tabular} \right\}
\end{split}
\end{equation}
which we denote by $\zO \mapsto \zO^\dagger$ below. 
The image of the map is the set of orbit closures which have non-trivial intersection with $U$.
\item There is a smooth affine variety $A$ and an isomorphism $\rep_Q(\bd) \times A \cong U \cap \rep_Q(\bd)^\dagger$ which restricts to an isomorphism
\begin{equation}\label{eq:mainthmiso}
\zO \times A \cong U\cap \zO^\dagger
\end{equation}
for each $\zO$ as above.
So, any smooth equivalence class of singularity occurring in $\zO$ also occurs in $\zO^\dagger$.
\item There is a homomorphism of equivariant Grothendieck groups 
\[
f\colon K_{T}(K\backslash G){\to} K_{T(\bd)}(\rep_Q(\bd))
\]
such that $[\cO_{\zO}]=f([\cO_{\zO^\dagger}])$, where $T$ is the maximal torus of diagonal matrices in $G$, and $T(\bd)$ is the maximal torus of matrices which are diagonal in each factor of $GL(\bd)$.
\end{enumerate}
\end{theorem}

The following result of Bobi\'nski and Zwara becomes an immediate corollary of our Theorem \ref{thm:mainTheorem} (ii) and general results about spherical varieties (see \cite[Theorem~0.1]{Brion03} and \cite[Theorem 4.4.7]{Perrin14}); in fact, their statement about rational singularities in characteristic zero upgrades to a rational resolution in arbitrary characteristic using \cite[Thm.~1]{AP16}. 

\begin{corollary}[cf. Bobi\'nski--Zwara \cite{BZ02}]\label{cor:BZ}
Orbit closures in type $D$ quiver representation varieties are normal, Cohen-Macaulay, and have a rational resolution (in arbitrary characteristic).
\end{corollary}

Establishing a purely combinatorial model for $GL(\bd)$-orbit closure containment, also called \emph{degeneration order}, is open for type $D$ quivers. 
On the other hand, the poset of $B$-orbit closures in $K\backslash G$ has been completely described by Wyser in \cite{Wyser16} in terms of combinatorial objects called $(a,b)$-\emph{clans}\footnote{Wyser describes the Bruhat order for $K$-orbit closures in $G/B$, but this poset is isomorphic to the poset of $B$-orbit closures in $K\backslash G$ (both of which are isomorphic to the poset of $K\times B$-orbit closures in $G$).}; these are particular strings of length $a+b$ in elements of $\mathbb{N}\cup\{+, -\}$ which parametrize $B$-orbit closures \cite[Thm.~4.1]{ToshihikoToshio} (see also \cite[Thm.~2.2.8]{Yamamoto}). Theorem \ref{thm:mainTheorem} thus immediately implies:
 
 \begin{corollary}\label{cor:clans}
Retaining the notation of Theorem \ref{thm:mainTheorem}, the poset of orbit closures in a type $D$ quiver representation variety $\rep_Q(\bd)$, partially ordered by inclusion, is isomorphic to a subposet of the poset of $(a,b)$-clans. The elements in the subposet are exactly the $(a,b)$-clans that index $B$-orbit closures which are maximal in their $P$-orbit closure, and have non-trivial intersection with $U$.  
\end{corollary}
 
The poset of $B$-orbit closures in a double Grassmannian can be embedded in a poset of clans, by Corollary \ref{cor:clans} together with Theorem \ref{thm:GKinD} (i).
See Smirnov's works \cite{Smirnov08, SmirnovThesis} or his survey \cite[Section 4]{SmirnovSurvey} for existing results on orders of $B$-orbit closures in double Grassmannians.
We plan to further explore the combinatorics of orbit closure containment in both type $D$ representation varieties and double Grassmannians in future work. 

We now outline our proof method for Theorem \ref{thm:mainTheorem}. 
We begin in Section \ref{sect:prelim} by using an arrow contraction operation to show that the conclusion of Theorem \ref{thm:mainTheorem} is true for a given type $D$ representation variety $\rep_Q(\bd)$ if and only if it is true upon replacing 
$\rep_Q(\bd)$ by a particular open subvariety $X_Q$ of a representation variety $\rep_{Q^*(n)}(\bd^*)$ where $Q^*(n)$ is the following bipartite quiver (not of type $D$) 
\begin{equation}\label{eq:Dnquiver}
Q^*(n) :=\quad
\vcenter{\hbox{\begin{tikzpicture}[point/.style={shape=circle,fill=black,scale=.5pt,outer sep=3pt},>=latex]
   \node[outer sep=-2pt] (x_0) at (-2,0) {${x_0}$};
   \node[outer sep=-2pt] (x) at (-1,-0.6) {${y'_0}$};
  \node[outer sep=-2pt] (y) at (-1,0.6) {${y_0}$};
   \node[outer sep=-2pt] (1) at (0,0) {${x_1}$};
  \node[outer sep=-2pt] (2) at (1.5,0) {${y_1}$};
   \node[outer sep=-2pt] (3) at (3,0) {${x_2}$};
   \node[outer sep=-2pt] (d) at (3.75,0) {${\cdots}$};
   \node[outer sep=-2pt] (4) at (4.5,0) {${x_n}$};
  \node[outer sep=-2pt] (5) at (6,0) {${y_n}$};
  \path[->]
        (y) edge node[above] {${\beta_0}$} (x_0) 
	(y) edge node[auto] {${\za_1}$} (1)
  	(x) edge node[below] {${\za'_1}$} (1) 
  	(2) edge node[auto] {${\zb_1}$} (1) 
	(2) edge node[below] {${\za_2}$} (3) 
	(5) edge node[auto] {${\zb_n}$} (4); 
	   \end{tikzpicture}}}.
\end{equation}
This is the content of Section \ref{sect:bundles} and Lemma \ref{lem:QtoQn}, taken together. 

In Section \ref{sect:quiverRanks}, we describe a finite collection of $\mathbb{Z}$-valued functions on each representation variety $\rep_{Q^*(n)}(\bd^*)$. We then characterize orbits in associated type $D$ representation varieties in terms of these functions (Proposition \ref{prop:quiverranksorbit}).

In Section \ref{sect:slice}, we consider intersections of certain orbit closures in $K\backslash G$ with a locally closed subvariety $S \subset K \backslash G$. 
The main result of this section (Theorem \ref{thm:intersectS}) says that intersecting one of these orbit closures in $K\backslash G$ with $S$ is, up to a smooth factor, the same as taking an open subset of that orbit closure.  Working in $S$ is a very useful intermediate reduction for comparing orbit closures in $K \backslash G$ with those in $\rep_Q(\bd)$. We compare our result to the literature at the beginning of Section \ref{sect:slice}.

In Section \ref{sect:embedding}, we give an explicit closed embedding $\eta\colon \rep_{Q^*(n)}(\bd^*) \to S$. This embedding is equivariant with respect to an appropriate map of tori.

 In Section \ref{sect:Smirnov}, we use the functions from Section \ref{sect:quiverRanks} to characterize intersections $S\cap O_{\eta(V)}$ and $S\cap \overline{O_{\eta(V)}}$. Here $\eta$ is as above, and $O_{\eta(V)}$ (respectively $\overline{O_{\eta(V)}}$) denotes the $P$-orbit through $\eta(V)$ (respectively its closure in $K\backslash G$) where $P$ is the parabolic subgroup from Theorem \ref{thm:mainTheorem}.
 
We finally prove Theorem \ref{thm:mainTheorem} in Section \ref{sect:mainTheorem}: we relate orbit closures in type $D$ representation varieties to orbit closures in $X_Q$ via the material in Section \ref{sect:prelim}. We then relate orbit closures in $X_Q$ to (open subvarieties of) intersections of the form $S\cap \overline{O_{\eta(V)}}$, $V\in X_Q$, using the map $\eta$ defined in Section \ref{sect:embedding}, together with the functions analyzed in Sections \ref{sect:quiverRanks} and \ref{sect:Smirnov}. Finally, we use our work from Section \ref{sect:slice} to relate each intersection $S\cap \overline{O_{\eta(V)}}$ to an open subvariety of $\overline{O_{\eta(V)}}$ in $K\backslash G$.

We conclude the paper with an application to Frobenius splitting of varieties in Section \ref{sect:Frobenius}.  Here we give an example showing that simultaneous compatible Frobenius splitting of all orbit closures in type $D$ quiver representation varieties is not generally possible, in contrast to the case of type $A$ quivers \cite[Prop.~6.3]{KR}.  This and our main theorem imply that simultaneous compatible Frobenius splitting of all $B$-orbit closures in double Grassmannians is generally not possible either, answering a question raised in \cite[Rmk.~4.5]{AP16}.

\subsection{Relation to existing literature}\label{sect:literature}
We now recall existing connections between representation varieties of Dynkin quivers and (multiple) flag varieties.  

\subsubsection*{Type $A$ quivers} The connections between orbit closures in type $A$ quiver representation varieties and Schubert varieties in type $A$ flag varieties have a long history (not all in the language of quiver representations). 
For example, in the works \cite{MR717615,Zgradednilp,LMdegen,BZ01,BZ02, KMS}, 
solutions to various problems about type $A$ quiver orbit closures were shown to be governed by solutions to the corresponding problems for Schubert varieties. 
Bobi\'nski and Zwara \cite{BZ01, BZ02} completed the story for singularities (Problem (2)) in type $A$ by showing that the 
singularity types appearing in type $A$ quiver orbit closures exactly coincide with singularity types of Schubert varieties in type $A$ flag varieties. 

In \cite{KR}, the authors of the present paper proved an analogue of Theorem \ref{thm:mainTheorem} relating orbit closures in (arbitrary) type $A$ representation varieties to Schubert varieties in type $A$ flag varieties. This work was inspired by previous work of Zelevinsky \cite{Zgradednilp} and Lakshmibai-Magyar \cite{LMdegen} on a specific orientation. The main theorem of \cite{KR} was later used to produce explicit combinatorial formulas for equivariant $K$-classes and cohomology classes of orbit closures in \cite{KKR}. Other works on these classes include \cite{BFchernclass, MR1932326,FRdegenlocithom,BFR,KMS, yongComponent, MR2137947, MR2114821, MR2306279}.

\subsubsection*{Type $D$ quivers}
Bobi\'nski and Zwara \cite{BZ02} showed that every singularity appearing in an orbit closure of a type $D$ quiver representation variety (for arbitrary orientation) also appears in some $B$-orbit closure in a symmetric variety $G/K$, thus obtaining Corollary \ref{cor:BZ} via Brion's results cited above.
Their method involved an innovative use of Auslander-Reiten quivers to construct Hom-controlled functors between categories of representations of various type $D$ quivers, combined with Zwara's work \cite{Zwara02}.  Our work gives a proof which is accessible without specialized knowledge of finite-dimensional algebras.

Combinatorics of orbit closure containment for type $D$ quivers seem to first appear in the work of Abeasis-del Fra \cite{abeasisdelfraDn}.  Their (92 page) paper treats only the special orientation where all the arrows point towards the end of the long branch, presenting a graphic description of the covering relations via a series of tables.  A characterization of orbit closure containment in terms of dimensions of Hom spaces was later given by Bongartz \cite{Bongartz96} for all (extended) Dynkin quivers, and by Zwara \cite{Zwara99} for arbitrary representation-finite algebras.
There is also a cohomological characterization of orbit closure containment due to Feh\'er and Patakfalvi \cite{FP09}, valid for quivers of arbitrary Dynkin type.
In contrast, the connection with clans established in this work allows for an entirely self-contained combinatorial model for the orbit closure poset of arbitrary type $D$ quivers which does not require performing any representation theoretic calculations.  We plan to make this more explicit in an upcoming work.

There are few explicit formulas for equivariant $K$-classes of orbit closures outside of type $A$ quivers.
Buch's work \cite{Buch08} gives some guidance in what kind of formulas to look for, proving for all Dynkin quivers that these classes can be expressed as finite sums of products of stable double Grothendieck polynomials, with integer coefficients.  These coefficients are uniquely determined in his expression, and he furthermore conjectures positivity and alternating sign properties for these coefficients \cite[Conj.~1.1]{Buch08}.   One of the goals of this work is to provide a new tool for attacking this conjecture.

The only explicit formulas that we are aware of in the literature appear in works of Allman and Rim\'anyi, and are valid for arbitrary Dynkin quivers.  Allman's formula in \cite{allman1} calculates equivariant $K$-classes of orbit closures as iterated residues of certain rational functions (see \cite{Rimanyi14} for the cohomological case).
Rim\'anyi produced formulas in \cite{Rimanyi20} for the equivariant motivic Chern classes of orbit closures in Dykin type; this is a deformation from which the equivariant Grothendieck class can be recovered.

\subsection*{Remark on the field $k$}
Everywhere but Section \ref{sect:Frobenius}, we work over an algebraically closed field $k$ of arbitrary characteristic which is often omitted from the notation. Since all schemes we consider are defined over $\ZZ$, many of our results generalize to arbitrary infinite $k$ though.

\subsection*{Acknowledgements}  We thank Michel Brion for pointing out that the image of our embedding $\eta$ was in $K \backslash G$; this observation led us to consider the interplay between type $D$ quiver representation varieties, double Grassmannians, and $K\backslash G$, rather than just between the first two.  We also thank Allen Knutson for helpful discussions early in the project, Alex Woo and Alex Yong for pointing us to the works \cite{WooWyserYong} and \cite{KnutsonWooYong}, and Richard Rim\'anyi for his comments.
Finally, we thank an anonymous referee for a very careful reading and helpful comments.

\section{Preliminaries}\label{sect:prelim}
\subsection{Quivers and representation varieties}\label{sect:quivers}
Fix a quiver $Q$ with vertex set $Q_0$ and arrow set $Q_1$, and let $ta$ and $ha$ denote the \emph{tail} and \emph{head} of an arrow $ta \xrightarrow{a} ha$.
We say that a quiver $Q'$ is a \emph{subquiver} of $Q$ if $Q_0'\subseteq Q_0$, $Q_1'\subseteq Q_1$, and the tail and head functions for $Q'$ are restrictions of the tail and head functions for $Q$.
We use the term \emph{tree quiver} to refer to a connected quiver with $n$ vertices and $n-1$ arrows.

Given a \emph{dimension vector} $\bd\colon Q_0 \to \mathbb{Z}_{\geq 0}$, we have the \emph{representation space}
\begin{equation}
\mathtt{rep}_Q(\bd) := \prod_{a \in Q_1} \Mat(\bd(ta), \bd(ha)),
\end{equation}
where $\Mat(m, n)$ denotes the algebraic variety of matrices with $m$ rows, $n$ columns, and entries in $k$. Each $V=(V_a)_{a \in Q_1}$ in $\mathtt{rep}_Q(\bd)$ is a \emph{representation} of $Q$; each matrix $V_a$ maps row vectors in $k^{\bd(ta)}$ to row vectors in $k^{\bd(ha)}$ by right multiplication.  
We denote the total dimension of the dimension vector $\bd$ by $d = \sum_{z \in Q_0} \bd(z)$.  
There is a \emph{base change group} $GL(\bd) := \prod_{z \in Q_0} GL(\bd(z))$,
which acts on $\mathtt{rep}_Q(\bd)$. 
Here $GL(\bd(z))$ denotes the general linear group of invertible $\bd(z)\times \bd(z)$ matrices with entries in $k$.
Explicitly, if $g = (g_z)_{z\in Q_0}$ is an element of $GL(\bd)$, and $V = (V_a)_{a\in Q_1}$ is an element of $\mathtt{rep}_Q(\bd)$, then the (right) action of $GL(\bd)$ on $\mathtt{rep}_Q(\bd)$ is given by
$V\cdot g = (g_{ta}^{-1}V_a g_{ha})_{a\in Q_1}.$ 
We denote by $T(\bd) \leq GL(\bd)$ the maximal torus of matrices which are diagonal in each factor.
The closure of a $GL(\bd)$-orbit in $\mathtt{rep}_Q(\bd)$ is called a \emph{quiver locus}.  
We use the notation $\zO$ to denote a quiver locus in $\rep_Q(\bd)$, and $\zO^\circ$ the dense $GL(\bd)$-orbit of $\zO$. 
An introduction to the theory of quiver representations can be found in the texts \cite{assemetal,Schiffler:2014aa,DWbook}.

\subsection{Equivariant $K$-theory}\label{sect:Ktheory}
For equivariant $K$-theory, we follow the notation and terminology of \cite[\S\S3,\,4]{Buch08}.  More detailed references for the basic notions are \cite{Thomason88} , or \cite[\S5]{CGbook} over $\CC$.
If we have an algebraic group $G$ and an algebraic $G$-scheme $X$ (both over $k$), then we denote by $K_G(X)$ the Grothendieck group of the category of $G$-equivariant coherent sheaves. If $Y$ is a $G$-stable closed subscheme of $X$, we let $[\mathcal{O}_Y]\in K_G(X)$ denote the class of its structure sheaf.

When $G$ is reductive (e.g., $G=GL(\bd)$) and $T \subseteq G$ a maximal torus, restriction of the group action induces a monomorphism $K_G(X) \into K_T(X)$ \cite[Thm.~1.13]{Thomason88}.  With this in mind, we can study a class $[\mathcal{O}_Y]\in K_G(X)$ just as well via its image in $K_T(X)$.  In certain instances it is easier to consider $T$-actions than $G$-actions, for example in Theorem \ref{thm:intersectS}.

The following lemma is a geometric restatement of the standard duality between representations of $Q$ and its \emph{opposite quiver} $Q^{\text{op}}$ (the quiver obtained from $Q$ by reversing the direction of each arrow). It shows that it is essentially the same to understand the equivariant geometry of representation varieties for a quiver or its opposite quiver.
\begin{lemma}\label{lem:Qop}
Let $Q$ be any quiver and $\bd$ a dimension vector for $Q$.  Then the transpose map $\rep_Q(\bd) \to \rep_{Q^{\textup{op}}}(\bd)$ given by $V = (V_a)\mapsto V^T = (V_a^T)$ is an isomorphism which is equivariant with respect to the automorphism of $GL(\bd)$ given by $(g_z) \mapsto ((g_z^{T})^{-1})$. 
In particular, the transpose map induces an isomorphism of $GL(\bd)$-equivariant Grothendieck groups $K_{GL(\bd)}(\rep_{Q^{\textup{op}}}(\bd)) \xto{\sim} K_{GL(\bd)}(\rep_Q(\bd))$.
\end{lemma}

\subsection{Homogeneous fiber bundles}\label{sect:bundles}
Let $G$ be an algebraic group, $H\leq G$ a closed and connected algebraic subgroup, and $X$ a quasiprojective $H$-variety.
Write $X *_H G$ for the quotient of $X \times G$ by the free right action of $H$ given by $(x,g) \cdot h = (x\cdot h, h^{-1} g)$. 
This quotient, called an \emph{induced space} or \emph{homogeneous fiber bundle}, is a $G$-variety for the action $(x*g') \cdot g = x*g'g$.
One may consult \cite[\S5.2.18]{CGbook} or \cite[\S14.6]{FSR17} as general references, for example.
The following lemma gives a criterion for a homogeneous fiber bundle to decompose as a product of varieties.

\begin{lemma}\label{lem:KXHG}\cite[Lemma 14.6.9]{FSR17}
Let $G$ be an affine algebraic group and $H\leq G$ a closed subgroup. Let $X$ be a $G$-variety (where the $G$-action extends the $H$-action) and endow $X\times H\backslash G$ with the diagonal right $G$-action. Then there exists a $G$-equivariant isomorphism $\Gamma: X *_H G\rightarrow  X\times H\backslash G$. 
\end{lemma}

The $X$, $H$, $G$ in this paper always satisfy the hypotheses of \cite[Thm.~14.6.5]{FSR17}, so that $X \times G \to X*_H G$ is a geometric quotient, and the map $X \to X *_H G$ sending $x \mapsto x*1$ is a closed immersion.  In this way, we identify $X \subseteq X*_H G$ when taking intersections of $X$ with subvarieties of $X*_H G$.
The following results connect the problems we study for the $H$-variety $X$ to those for the $G$ variety $X *_H G$.

\begin{proposition}\label{prop:fiberbundle}
In the general setup of this section, all of the following hold.
\begin{enumerate}[(i)]
\item The following association is a bijection.  
\begin{equation}\label{eq:GHvarieties}
\begin{split}
\left\{\begin{tabular}{c} $G$-stable subvarieties\\ of $X *_H G$\end{tabular} \right\}
& \to
\left\{\begin{tabular}{c}  $H$-stable subvarieties\\ of $X$ \end{tabular} \right\}
\\
Y \qquad \qquad &\mapsto \qquad \qquad Y \cap X\\
Z \cdot G \qquad \qquad &\mapsfrom \qquad \qquad Z
\end{split}
\end{equation}
In particular, it restricts to a bijection on orbits and induces an isomorphism of orbit closure posets.
\item Let $Z \subseteq X$ be an $H$-stable subvariety. Any smooth equivalence class of singularity that occurs in $Z$ also occurs in $Z\cdot G$. 
\item Restriction along $X \into X *_H G$ induces an isomorphism of equivariant Grothendieck groups $K_G(X *_H G) \simeq K_H(X)$.
\end{enumerate}
\end{proposition}
\begin{proof}
(i) To see that the maps in \eqref{eq:GHvarieties} are mutually inverse bijections, one can directly check that $(Y \cap X)\cdot G = Y$ and $(Z \cdot G) \cap X = Z$ (see \cite[Lem.~14.6.7]{FSR17}).\\ 
(ii) Since $Z$ is $H$-stable, we have $Z \cdot G = Z *_H G$. 
Given $z \in Z$ we will show that the point $z*1 \in Z\cdot G$ is in the same smooth equivalence class of singularities,  following the definition from \cite[Section 1.7]{Hesselink76}.  
That is, we will show that there is a scheme $T$, a point $t\in T$, and smooth $k$-morphisms $f_1: T\rightarrow Z$ and $f_2: T\rightarrow Z\cdot G$ such that $f_1(t) = z$ and $f_2(t) = z*1$. To this end, let $T = Z\times G$ and $t = (z,1)$.  Let $f_1: Z\times G\rightarrow Z$ denote the projection map onto the first factor and let $f_2: Z \times G \to Z*_H G$ be the quotient map. Then $f_1(z,1) = z$ and $f_2(z,1) =z*1$.
The projection map $f_1$ is smooth since $G$ is, so it remains to see that the quotient map $f_2$ is smooth.  This follows from the fact that it is a geometric quotient by the free action of $H$, so each fiber is isomorphic to the smooth variety $H$ and the map is flat \cite[Prop.~0.9]{MFK-GIT}.

(iii) This is \cite[Thm.~1.10]{Thomason88}.
\end{proof}  
  
\subsection{Double Grassmannians as special cases of type $D$ quiver representation varieties}\label{sect:GrassinD}
We now show how Problems (1), (2), (3) of the introduction, in the case of double Grassmannians, can be seen as special cases of the corresponding problems for type $D$ quiver representation varieties.  While the general principle at use is already known (parts (i) and (ii) are essentially proven in \cite{BZ02}), at least the $K$-theory part does not seem to appear elsewhere in the literature, and so we include a complete proof.

\begin{theorem}\label{thm:GKinD}
Given any double Grassmannian $X=\Gr_a(k^n) \times \Gr_b(k^n)$,
there exists a type $D$ quiver $Q$ with dimension vector $\bd$, and a $GL(\bd)$-stable open subvariety $U \subset \rep_Q(\bd)$ such that the following hold.
\begin{enumerate}[(i)]
\item There is an injective, order preserving map of partially ordered sets
\begin{equation}\label{eq:Xtorep}
\begin{split}
\left\{\begin{tabular}{c} $B$-orbit closures\\ in $X$ \end{tabular} \right\}
& \to
\left\{\begin{tabular}{c} $GL(\bd)$-orbit closures\\ in $\rep_Q(\bd)$\end{tabular} \right\}
\end{split}
\end{equation}
which we denote by $\mathscr{O} \mapsto \mathscr{O}'$ below.
The image of the map is the set of orbit closures which have non-trivial intersection with $U$.
\item Any smooth equivalence class of singularity occurring in $\mathscr{O}$ also occurs in $\mathscr{O}'$.
\item There is a homomorphism of equivariant Grothendieck groups 
\[
f\colon K_{GL(\bd)}(\rep_Q(\bd))\to K_B(X)
\]
such that $[\cO_{\mathscr{O}}]=f([\cO_{\mathscr{O}'}])$.
\end{enumerate}
\end{theorem}

\begin{proof}
Let $X=\Gr_a(k^n) \times \Gr_b(k^n)$ be the given double Grassmannian.
Let $Q$ be the type $D$ quiver with $n+2$ vertices and all arrows oriented towards the branch point, and let $\bd$ be the dimension vector shown below.
\begin{equation}\label{eq:Dnintro}
(Q, \bd) =\quad
\vcenter{\hbox{\begin{tikzpicture}[point/.style={shape=circle,fill=black,scale=.5pt,outer sep=3pt},>=latex]
  \node[outer sep=-2pt] (y) at (-1.5,0.6) {$a$};
   \node[outer sep=-2pt] (x) at (-1.5,-0.6) {$b$};
   \node[outer sep=-2pt] (1) at (0,0) {$n$};
  \node[outer sep=-2pt] (2) at (2,0) {$n-1$};
   \node[outer sep=-2pt] (3) at (4,0) {$n-2$};
   \node[outer sep=-2pt] (3a) at (5,0) {};
   \node[outer sep=-2pt] (d) at (5.5,0) {${\cdots}$};
   \node[outer sep=-2pt] (4a) at (6,0) {};
   \node[outer sep=-2pt] (4) at (7,0) {$2$};
  \node[outer sep=-2pt] (5) at (8,0) {$1$};
  \path[->]
	(y) edge node[auto] {} (1)
  	(x) edge node[below] {} (1) 
  	(2) edge node[auto] {} (1) 
	(3) edge node[below] {} (2) 
	(3a) edge node[below] {} (3) 
	(4) edge node[below] {} (4a) 
	(5) edge node[auto] {} (4); 
	   \end{tikzpicture}}}
\end{equation}
Here, the representation variety is the product of matrix spaces
\begin{equation}\label{eq:repqd1}
\rep_Q(\bd) = \Mat(a,n) \times \Mat(b,n) \times \prod_{i=1}^{n-1} \Mat(i, i+1)
\end{equation}
carrying the natural right action of 
\begin{equation}
GL(\bd) = GL(a) \times GL(b) \times \prod_{i=1}^{n} GL(i).
\end{equation}

Consider the open subvariety $U \subseteq \rep_Q(\bd)$ where all matrices have maximal rank, and decompose
$GL(\bd) = G' \times G$ where  $G':=GL(a) \times GL(b) \times \prod_{i=1}^{n-1} GL(i)$ and $G=GL(n)$.
Then taking the quotient of the action of $G'$ on $U$, the first two factors of \eqref{eq:repqd1} become $X$ and the last product of factors becomes the variety of complete flags in $k^{n}$, which we identify as $B\backslash G$ where $B$ is a Borel subgroup of $G$.
Thus, we get an isomorphism of the quotient variety
$U/G' \simeq X \times (B\backslash G)$.
This isomorphism is $G$-equivariant with respect to the unused $G$-action on $U/G'$, which translates to the right action $(M,N,Bg)\cdot h=(Mh,Nh,Bgh)$ for $h \in G$.  By Lemma \ref{lem:KXHG}, we have a $G$-equivariant isomorphism $X \times (B \backslash G) \simeq X *_B G$. 

Applying Proposition \ref{prop:fiberbundle}(i) to the $G$-equivariant isomorphism $U/G'\cong X*_B G$, we get an inclusion preserving bijection between the set of $B$-orbit closures in $X$ and the set of $G$-orbit closures in $U/G'$. Part (i) of the present theorem now follows because there is an inclusion preserving bijection between the set of $GL(\bd)$-orbit closures in $\rep_Q(\bd)$ which have non-trivial intersection $U$, and the set of $G$-orbit closures of $U/G'$.

Similarly, Proposition \ref{prop:fiberbundle}(ii) applied to the $G$-equivariant isomorphism $U/G'\cong X*_B G$ shows that any equivalence class of singularity occuring in a $B$-orbit closure $\mathscr{O}\subseteq X$ occurs in a corresponding $G$-orbit closure $\mathscr{O}'' \subseteq U/G'$. Consider the pullback diagram below.
\begin{equation}\label{eq:cartesianDiagram}
\begin{tikzcd}
\mathscr{O}''\times_{U/G'} U \arrow[r] \arrow[d]
& U \arrow[d] \\
\mathscr{O}'' \arrow[r]
& U/G'
\end{tikzcd}
\end{equation}
Since the morphism $U\rightarrow U/G'$ is locally of finite presentation with smooth fibers, and flat as well by \cite[Theorem 23.1]{Matsumura},  it is a smooth morphism by \cite[Lemma 29.33.3]{stacks-project}. Since smooth morphisms are stable under base change \cite[Lemma 29.33.5]{stacks-project}, the map $\mathscr{O}''\times_{U/G'} U \rightarrow \mathscr{O}''$ is smooth. Furthermore, since this morphism has a reduced target with reduced fibers, $\mathscr{O}''\times_{U/G'} U$ is reduced \cite[Theorem 23.9]{Matsumura}. Consequently, the irreducible, $GL(\bd)$-invariant subscheme $\mathscr{O}''\times_{U/G'} U$ must be the intersection of an orbit closure in $\mathscr{O}'\subseteq \rep_{Q}(\bd)$ with $U$. This proves part (ii) of the present theorem.

For part (iii) we take the composition of group homomorphisms
\[
f\colon K_{GL(\bd)}(\rep_Q(\bd)) \onto K_{GL(\bd)}(U) \simeq K_G(U/G') \simeq K_B(X), 
\]
where the quotient map is induced by restriction to the open set \cite[Cor.~2.4]{Thomason87}, the first isomorphism is \cite[Prop.~2.3]{merkurjev05}, and the second is from Proposition \ref{prop:fiberbundle} (iii) where the pair $H \leq G$ in the statement of Proposition \ref{prop:fiberbundle} is $B \leq G$ here.
\end{proof}

The following corollary is immediate by restricting to the open subvariety $K\backslash G \subset X$.
\begin{corollary}\label{cor:GKinD}
Let $G=GL(a+b)$ and $K=GL(a)\times GL(b)$ embedded diagonally in $G$.  All statements in Theorem \ref{thm:GKinD} remain true upon replacing $X$ by $K\backslash G$.
\end{corollary}

\begin{remark}\label{SchubToTypeA}
By an argument similar to the proof of Theorem \ref{thm:GKinD}, all statements in Theorem \ref{thm:GKinD} remain true upon replacing $X$ by a type $A$ flag variety $B\backslash G$, and replacing the type $D$ quiver and dimension vector of \eqref{eq:Dnintro} with a type $A$ quiver and dimension vector of the form
\begin{equation}
\vcenter{\hbox{\begin{tikzpicture}[point/.style={shape=circle,fill=black,scale=.5pt,outer sep=3pt},>=latex]
  \node[outer sep=-2pt] (2') at (-2,0) {$n-1$};
   \node[outer sep=-2pt] (3'a) at (-3,0) {};
   \node[outer sep=-2pt] (d') at (-3.5,0) {${\cdots}$};
   \node[outer sep=-2pt] (4'a) at (-4,0) {};
   \node[outer sep=-2pt] (4') at (-5,0) {$2$};
  \node[outer sep=-2pt] (5') at (-6,0) {$1$};
   \node[outer sep=-2pt] (1) at (-0.5,0) {$n$};
  \node[outer sep=-2pt] (3) at (1,0) {$n-1$};
   \node[outer sep=-2pt] (3a) at (2,0) {};
   \node[outer sep=-2pt] (d) at (2.5,0) {${\cdots}$};
   \node[outer sep=-2pt] (4a) at (3,0) {};
   \node[outer sep=-2pt] (4) at (4,0) {$2$};
  \node[outer sep=-2pt] (5) at (5,0) {$1$};
  \path[->]
  (5') edge node[auto] {} (4')
        (4') edge node[auto] {} (4'a)
        (3'a) edge node[auto] {} (2')
  	(2') edge node[auto] {} (1) 
	(3) edge node[auto] {} (1) 
	(3a) edge node[auto] {} (3) 
	(4) edge node[auto] {} (4a) 
	(5) edge node[auto] {} (4); 
	   \end{tikzpicture}}} \qedhere
\end{equation}
\end{remark}

\subsection{Arrow contraction}\label{sect:contract}
We recall the following common construction from graph theory to establish notation.

\begin{definition}
Let $Q$ be a quiver and $a \in Q_1$ not a loop.  The \emph{contraction of $Q$ along $a$} is the quiver $Q^a$ with vertex set $Q^a_0:=(Q_0 \setminus \{ta, ha\}) \amalg \{\star\}$ and arrow set $Q^a_1:=Q_1 \setminus \{a\}$.  
The tail and head functions of the contraction are given by restricting the tail and head functions of $Q$ to $Q^a_1$, while replacing both $ta$ and $ha$ with $\star$ in the codomain.
\end{definition}

We call $A \subseteq Q_1$ \emph{admissible for contraction} if the smallest subquiver of $Q$ containing $A$ is a disjoint union of tree quivers. 
Given such an $A$, we recursively define the \emph{contraction of $Q$ along $A$} by $Q^{\emptyset}:= Q$ and $Q^A:=(Q^{A\setminus \{a\}})^a$ for any $a \in A$.
It can be easily checked that $Q^A$ is independent of the order in which the contractions are taken, and the arrow being contracted is never a loop.

See Figure \ref{fig:contract} for an example.
Note that contraction of $Q$ along $a$ determines a map on vertices $\nu^a \colon Q_0 \to Q^a_0$ sending both $ta, ha$ to $\star$, and each vertex besides $ta, ha$ to itself.  Composing these we get $\nu^A \colon Q_0 \to Q^A_0$ for admissible $A \subseteq Q_1$.
Then a dimension vector $\bd$ for $Q^A$ determines a dimension vector $\bd^A$ for $Q$ by $(\bd^A)(z) = \bd(\nu^A(z))$ for all $z \in Q_0$.

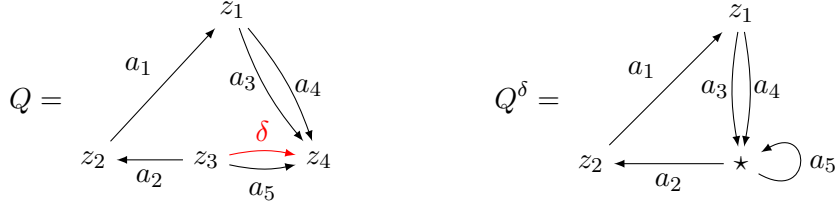
\begin{figure}
\[
Q=
\vcenter{\hbox{\begin{tikzpicture}[point/.style={shape=circle,fill=black,scale=.5pt,outer sep=3pt},>=latex]
   \node (1) at (0,0) {$z_2$};
   \node (3) at (1.85,2) {$z_1$};
   \node(4) at (1.5,0) {{$z_3$}};
\node (2) at (3,0) {{$z_4$}};
\draw[->, bend right = 12] (4) edge node[below = 0.05] {$a_5$} (2);
\draw[->, bend left =10] (3) edge node[left=0.25] {$a_3$}  (2);
\draw[->, bend right =10] (3) edge node[right=0.25] {$a_4$} (2);
\draw[->] (4) edge node[auto] {$a_2$} (1);
\draw[->] (1) edge node[auto] {$a_1$} (3);
\draw[->, bend left = 12] (4) edge[red] node[auto] {{\color{red}$\delta$}} (2);
   \end{tikzpicture}}}
\hspace{.8in}
Q^\delta=
\vcenter{\hbox{\begin{tikzpicture}[point/.style={shape=circle,fill=black,scale=.5pt,outer sep=3pt},>=latex]
   \node (1) at (0,0) {$z_2$};
   \node (3) at (2,2) {$z_1$};
\node (2) at (2,0) {$\star$} edge[in=30,out=-30,loop] node[right] {$a_5$} ();
\draw[->, bend left =10] (3) edge node[left=0.15] {$a_3$}  (2);
\draw[->, bend right =10] (3) edge node[right=0.15] {$a_4$} (2);
\draw[->] (2) edge node[auto] {$a_2$} (1);
\draw[->] (1) edge node[auto] {$a_1$} (3);
   \end{tikzpicture}}}
\]
\caption{The quiver on the right is obtained by contracting the arrow $\delta$ in the quiver on the left.}\label{fig:contract}
\end{figure}

Given a quiver $Q$ and set of arrows $A$ admissible for contraction, 
let $H^A$ be the following closed subgroup of $GL(\bd^A)$:
\begin{equation}\label{eq:HinGL}
H^A = \setst{(g_z)_{z \in Q_1} \in GL(\bd^A)}{g_{t(a)} = g_{h(a)}\ \forall a \in A}.
\end{equation}
There is a group isomorphism $\tau: \GL(\bd)\xrightarrow{\sim} H^A$, which maps $(g_z)_{z\in Q^A_0}$ to $(g_{\nu^A(z)})_{z\in Q_0}$. Then, with respect to the usual conjugation actions of $GL(\bd)$ on $\rep_{Q^A}(\bd)$ and $H^A$ on $\rep_Q(\bd^A)$ and the map $\tau$, we can equivariantly embed $\psi: \rep_{Q^A}(\bd) \hookrightarrow \rep_Q{(\bd^A)}$ as a closed subvariety by identifying $(V_a)_{a \in Q^A_1}$ with $(W_a)_{a \in Q_1}$ where $W_a = V_a$ if $a \notin A$, and $W_a = 1$ (an appropriately sized identity matrix) if $a \in A$. 
In what follows, we identify $\rep_{Q^A}(\bd)$ and its $GL(\bd)$-action with $\psi(\rep_{Q^A}(\bd))$ and its $H^A$-action, and consider the homogeneous fiber bundle $\rep_{Q^A}(\bd) *_{H^A} GL(\bd^A)$.

\begin{lemma}\label{lem:trivialBundleForMainTheorem}
The homogeneous fiber bundle $\rep_{Q^A}(\bd) *_{H^A} GL(\bd^A)$ is $GL(\bd^A)$-equivariantly isomorphic to $\rep_{Q^A}(\bd)\times H^A\backslash GL(\bd^A)$.
\end{lemma}

\begin{proof}
We will define a right action of $GL(\bd^A)$ on $\rep_{Q^A}(\bd)$ that restricts to the  $H^A$-action. The result then follows from Lemma \ref{lem:KXHG}. To define the desired $GL(\bd^A)$-action, it suffices to produce a group homomorphism $\phi: \GL(\bd^A)\rightarrow H^A$ satisfying $\phi(h) = h$ for each $h\in H^A$. 

Let $T \subseteq Q$ be the smallest subquiver containing $A$ and all vertices of $Q$.   
Let $T_1,\dots, T_r$ denote the connected components of $T$ and choose  a vertex $z_i\in T_i$ for each $1\leq i\leq r$.  
Define a group homomorphism $\phi: GL(\bd^A)\rightarrow H^A$ by $\phi((g_z)_{z\in Q_0}) = (h_z)_{z\in Q_0}$ where $h_z = g_{z_i}$ for the unique $i$ such that $z\in T_i$. Observe that $\phi(h) = h$ for each $h\in H^A$.
\end{proof}

\begin{proposition}\label{prop:doublebundle}
Let $A\subseteq Q_1$ be admissible for contraction and $\bd$ a dimension vector for $Q^A$.
Let $U \subseteq \rep_Q(\bd^A)$ where the matrix over each $a \in A$ is invertible.
Then there is a $GL(\bd^A)$-equivariant isomorphism $\rep_{Q^A}(\bd) *_{H^A} GL(\bd^A) \to U$. 
\end{proposition}

\begin{proof}
By Lemma \ref{lem:trivialBundleForMainTheorem}, it suffices to produce a $GL(\bd^A)$-equivariant isomorphism from $U$ to $\rep_{Q^A}(\bd)\times H^A\backslash GL(\bd^A)$. To begin, we will define a morphism 
\begin{equation}
e\colon U^A \to GL(\bd^A), \quad U^A := \prod_{a \in A} \Mat^\circ(\bd^A(ta), \bd^A(ha)),
\end{equation}
where $\Mat^\circ$ denotes the subset of matrices of full rank. Let $T \subseteq Q$ be the smallest subquiver containing $A$ and all vertices of $Q$, and let $T_1,\dots, T_r$ be its connected components, recalling that each is a tree since $A$ is admissible. Arbitrarily choose distinguished vertices $z_1,\dots, z_r$ with $z_i\in T_i$.
 For $z \in T_i$, define a sequence of symbols
\begin{equation}\label{eq:walk}
z_i\rightsquigarrow z := a_1^{\varepsilon_1} \cdots a_t ^{\varepsilon_t}
\end{equation}
where each $a_j \in A$ and $\varepsilon_j = \pm 1$.  If $z=z_i$, we take the sequence to be empty by convention, otherwise the fact that $T_i$ is a tree implies that there is a unique minimal sequence of arrows $(a_1, \dotsc, a_t)$ in $A$ such that walking along these arrows (with or against the direction of the arrows) takes one from $z_i$ to $z$ without leaving $T_i$. 
For each $j$, we set $\varepsilon_j = 1$ if $a_j$ is traversed in this walk with the direction of the arrow, and $\varepsilon_j = -1$ if $a_j$ is traversed against the direction of the arrow. Now to define the morphism $e$, given $V^A :=(V_a)_{a\in A} \in U^A$ and $z \in Q_0$ we associate an element $V^A(z) \in GL(\bd^A(z))$ as follows. Suppose $z\in T_i$. 
Recalling \eqref{eq:walk}, we set $V^A(z) = V_{a_1}^{\varepsilon_1} \cdots V_{a_t}^{\varepsilon_t}$, where $V^A(z)=1$ if $z=z_i$.  Note that this is indeed well-defined since each $V_{a_j} \in GL(\bd^A(z))$ by definition of $\bd^A$.  This defines $e$ by $e(V^A) = (V^A(z))_{z \in Q_0}$.

Using the above notation, and letting every appearance of a dot $\cdot$ below denote the usual action of $GL(\bd^A)$ on $\rep_{Q}(\bd^A)$, our desired isomorphism $E: U\rightarrow \rep_{Q^A}(\bd)\times H^A\backslash GL(\bd^A)$ is defined by
\[
V = (V_a)_{a\in Q_0}  \mapsto (V\cdot e(V^A)^{-1}, \enskip H^A e(V^A)),
\]
recalling that we have identified $\rep_{Q^A}(\bd)$ with a closed subvariety in $\rep_{Q}(\bd^A)$ (see the discussion preceding Lemma \ref{lem:trivialBundleForMainTheorem}). 
One may check from the definitions that this map is equivariant with respect to the natural right $GL(\bd^A)$-action on the source $U$, and the diagonal right $GL(\bd^A)$-action on the target, where $GL(\bd^A)$ acts on $\rep_{Q^A}(\bd)$ via the homomorphism $\phi: GL(\bd^A)\rightarrow H^A$ by $\phi((g_z)_{z\in Q_0}) = (h_z)_{z\in Q_0}$ where $h_z = g_{z_i}$, if $z\in T_i$, from Lemma \ref{lem:trivialBundleForMainTheorem}. Then, the (well-defined) inverse morphism to $E$ is given by
\[
E^{-1}: \rep_{Q^A}(\bd)\times H^A\backslash GL(\bd^A)\rightarrow U, \quad (W, H^Ag)\mapsto (W \cdot \phi(g)^{-1})\cdot g. \qedhere
\]
\end{proof}

We use the above results to reduce the problems we study for representation varieties of arbitrary type $D$ quivers (in particular, of any orientation) to the study of certain representation varieties for the specific quivers $Q^*(n)$ of \eqref{eq:Dnquiver}.
Let $(Q,\bd)$ be a type $D$ quiver with dimension vector.  We want to define an associated $(Q^*, \bd^*)$.
Let $\gamma, \delta \in Q_1$ be the arrows in the length 1 branches.  After perhaps replacing $Q$ with $Q^{\rm op}$,  we can assume at least one of $\gamma, \delta$ points towards the branch point (say $\delta$); by Lemma \ref{lem:Qop}, this does not affect the equivariant geometry.
It is visually apparent that we can take $n$ large enough so that $Q^*:=Q^*(n)$ contracts to $Q$ via some admissible set $A$ of arrows of $Q^*$, and this determines a dimension vector $\bd^A$ as defined above.
In the rest of the article, we denote the associated map on vertices by $\nu\colon Q^*_0 \to Q_0$, leaving the set of arrows implicit.

To define $\bd^*$, we distinguish two cases.  If $\gamma$ and $\delta$ are composable, take $\bd^*:=\bd^A$.
\begin{equation}\label{eq:fundorient1}
\vcenter{\hbox{\begin{tikzpicture}[point/.style={shape=circle,fill=black,scale=.5pt,outer sep=3pt},>=latex]
\node at (-1,1.5) {$\underline{(Q, \bd)}$};
  \node[outer sep=-2pt] (y) at (-1.5,0.6) {$\bd(1)$};
   \node[outer sep=-2pt] (x) at (-1.5,-0.6) {$\bd(2)$};
   \node[outer sep=-2pt] (1) at (0,0) {$\bd(3)$};
   \node[outer sep=-2pt] (2) at (1,0) {};
   \node (d) at (1.5,0) {${\cdots}$};
  \path[->]
	(1) edge node[above] {$\gamma$} (y)
  	(x) edge node[below] {$\delta$} (1);
\draw[-] (2) edge node[auto] {} (1);
		   \end{tikzpicture}}}
\qquad \qquad \qquad
\vcenter{\hbox{\begin{tikzpicture}[point/.style={shape=circle,fill=black,scale=.5pt,outer sep=3pt},>=latex]
\node at (-1,1.5) {$\underline{(Q^*, \bd^*)}$};
  \node[outer sep=-2pt] (y1) at (-3,0) {$\bd(1)$};
  \node[outer sep=-2pt] (y2) at (-1.5,0.6) {$\bd(3)$};
   \node[outer sep=-2pt] (x) at (-1.5,-0.6) {$\bd(2)$};
   \node[outer sep=-2pt] (1) at (0,0) {$\bd(3)$};
   \node[outer sep=-2pt] (2) at (1,0) {};
   \node (d) at (1.5,0) {${\cdots}$};
  \path[->]
	(y2) edge node[above] {$\gamma$} (y1)
  	(x) edge node[below] {$\delta$} (1);
\draw[->,dashed] (y2) edge node[above] {} (1);
\draw[-] (2) edge node[auto] {} (1);
		   \end{tikzpicture}}}
\end{equation}
If both $\gamma, \delta$ point towards the branch, however, we must take $\bd^*(x_0)=0$, and $\bd^*$ equal to $\bd^A$ at the other vertices.
\begin{equation}\label{eq:fundorient2}
\vcenter{\hbox{\begin{tikzpicture}[point/.style={shape=circle,fill=black,scale=.5pt,outer sep=3pt},>=latex]
\node at (-1,1.5) {$\underline{(Q, \bd)}$};
  \node[outer sep=-2pt] (y) at (-1.5,0.6) {$\bd(1)$};
   \node[outer sep=-2pt] (x) at (-1.5,-0.6) {$\bd(2)$};
   \node[outer sep=-2pt] (1) at (0,0) {$\bd(3)$};
   \node[outer sep=-2pt] (2) at (1,0) {};
   \node (d) at (1.5,0) {${\cdots}$};
  \path[->]
	(y) edge node[above] {$\gamma$} (1)
  	(x) edge node[below] {$\delta$} (1);
\draw[-] (2) edge node[auto] {} (1);
		   \end{tikzpicture}}}
\qquad \qquad \qquad
\vcenter{\hbox{\begin{tikzpicture}[point/.style={shape=circle,fill=black,scale=.5pt,outer sep=3pt},>=latex]
\node at (-1,1.5) {$\underline{(Q^*, \bd^*)}$};
  \node[outer sep=-2pt] (y1) at (-3,0) {0};
  \node[outer sep=-2pt] (y2) at (-1.5,0.6) {$\bd(1)$};
   \node[outer sep=-2pt] (x) at (-1.5,-0.6) {$\bd(2)$};
   \node[outer sep=-2pt] (1) at (0,0) {$\bd(3)$};
   \node[outer sep=-2pt] (2) at (1,0) {};
   \node (d) at (1.5,0) {${\cdots}$};
  \path[->]
	(y2) edge node[above] {} (y1)
  	(x) edge node[below] {$\delta$} (1);
\draw[->] (y2) edge node[above] {$\gamma$} (1);
\draw[-] (2) edge node[auto] {} (1);
		   \end{tikzpicture}}}
\end{equation}

The following lemma is now an immediate consequence of the above, Proposition \ref{prop:doublebundle}, and Lemma \ref{lem:trivialBundleForMainTheorem}.

\begin{lemma}\label{lem:QtoQn}
Let $(Q,\bd)$ be a type $D$ quiver with dimension vector, and  $(Q^*, \bd^*)$ defined immediately above.
Then there is a $GL(\bd^*)$-stable, open subvariety $X_Q \subseteq \rep_{Q^*}(\bd^*)$ and 
 $GL(\bd^*)$-equivariant isomorphisms 
 \[X_Q \simeq \rep_Q(\bd) *_{GL(\bd)} GL(\bd^*) \simeq \rep_Q(\bd) \times (GL(\bd)\backslash GL(\bd^*)), \]
 where $GL(\bd)$ is identified as a subgroup $H^A\leq GL(\bd^*)$ as in \eqref{eq:HinGL}, that is,
 \begin{equation}
H^A = \setst{(g_z)_{z \in Q_1} \in GL(\bd^*)}{g_{t(a)} = g_{h(a)}\ \forall a \in A\setminus \beta_0}.
\end{equation}
\end{lemma}

\begin{example}\label{eg:extension1}
If $Q$ is the quiver
\[
Q=\qquad\vcenter{\hbox{\begin{tikzpicture}[yscale=0.6,point/.style={shape=circle,fill=black,scale=.5pt,outer sep=3pt},>=latex]
  \node[outer sep=-2pt] (1) at (-1,3) {$1$}; 
   \node[outer sep=-2pt] (2) at (-1,1) {$2$};
   \node[outer sep=-2pt] (3) at (0,2) {$3$};
   \node[outer sep=-2pt] (4) at (1.25,2) {$4$};
  \node[outer sep=-2pt] (5) at (2.5,2) {$5$};
   \node[outer sep=-2pt] (6) at (3.75,2) {$6$};
  \node[outer sep=-2pt] (7) at (5,2) {$7$};
  \path[->]
        (3) edge node[above, pos=0.3] {$\zg$} (1) 
  	(2) edge node[below,pos=0.3] {$\zd$} (3) 
	(3) edge node[below,pos=0.3] {} (4)
  	(4) edge node[below,pos=0.3] {} (5) 
	(6) edge node[below,pos=0.3] {} (5)
	(7) edge node[below] {} (6);
   \end{tikzpicture}}}, 
\]
then $Q^*(4)$ shown below admits a contraction to $Q$, where contracted arrows are dashed.
\[
Q^*(4) =\quad
\vcenter{\hbox{\begin{tikzpicture}[point/.style={shape=circle,fill=black,scale=.5pt,outer sep=3pt},>=latex]
   \node[outer sep=-2pt] (x_0) at (-2,0) {${x_0}$};
   \node[outer sep=-2pt] (x) at (-1,-0.6) {${y'_0}$};
  \node[outer sep=-2pt] (y) at (-1,0.6) {${y_0}$};
   \node[outer sep=-2pt] (1) at (0,0) {${x_1}$};
  \node[outer sep=-2pt] (2) at (1.5,0) {${y_1}$};
   \node[outer sep=-2pt] (3) at (3,0) {${x_2}$};
   \node[outer sep=-2pt] (4) at (4.5,0) {${y_2}$};
   \node[outer sep=-2pt] (5) at (6,0) {${x_3}$};
  \node[outer sep=-2pt] (6) at (7.5,0) {${y_3}$};
  \node[outer sep=-2pt] (7) at (9,0) {${x_4}$};
  \node[outer sep=-2pt] (8) at (10.5,0) {${y_4}$};
  \path[->]
        (y) edge node[above] {${\beta_0}$} (x_0) 
	(y) edge[red,dashed] node[auto] {${\za_1}$} (1)
  	(x) edge node[below] {${\za'_1}$} (1) 
  	(2) edge[red,dashed] node[below] {${\zb_1}$} (1) 
	(2) edge node[below] {${\za_2}$} (3) 
	(4) edge[red,dashed] node[below] {${\zb_2}$} (3)
	(4) edge node[below] {${\za_3}$} (5)
	(6) edge node[below] {${\zb_3}$} (5)
	(6) edge[red,dashed] node[below] {${\za_4}$} (7)
	(8) edge node[below] {${\zb_4}$} (7);
	   \end{tikzpicture}}}
\]
If $\bd$ is a dimension vector for $Q$, then $X_Q$ is the open subvariety of $\rep_{Q^*(4)}(\bd^*)$ where the maps over arrows $\alpha_1, \beta_1, \beta_2, \alpha_4$ are invertible.
\end{example}

\section{Quiver rank functions}\label{sect:quiverRanks}
We now define some collections of matrices associated to quiver representations, whose ranks will be used to describe orbits and orbit closures in quiver representation varieties.

\subsection{Preliminary matrix definitions}\label{sect:prelimMatrixDefs}
The matrices in this paper usually have their rows and columns partitioned into blocks labeled by vertices of some quiver (with repetition allowed).  We follow \cite[\S4]{RZ13} to formalize this.

\begin{definition}
Let $Q$ be a quiver and $\br=(r_1, \dotsc, r_m)$, $\bc=(c_1, \dotsc, c_n)$ be two sequences in $Q_0$.  A \emph{$Q$-matrix} is a matrix such that:
\begin{enumerate}[(QM1)]
\item its rows are labeled from top to bottom by $\br$;
\item its columns are labeled from left to right by $\bc$;  
\item the entry in row $i$ and column $j$ is a formal $k$-linear combination of paths in $Q$ from the vertex $r_i$ to the vertex $c_j$ (with length 0 trivial paths allowed).
\end{enumerate}

Let $\Phi$ be a $Q$-matrix and $V=(V_a) \in \rep_Q(\bd)$.  Consider an entry $\sum_l \lambda_l p_l$ of $\Phi$, where each $\lambda_l \in k$ and each $p_l$ is a path in $Q$ with the same source and target.
The \emph{evaluation of $\Phi$ at $V$}, denoted $\Phi_V$, is the $\sum_i \bd(r_i) \times \sum_j \bd(c_j)$ matrix over $k$ obtained by replacing each entry as above with the corresponding linear combination of matrices $\sum_l \lambda_l V_{p_l}$.
\end{definition}

Note that $\Phi_V$ is naturally equipped with a block subdivision with block row $r_i$ of height $\bd(r_i)$ and block column $c_j$ of width $\bd(c_j)$.

For an arbitrary block matrix $M$ with block rows labeled from top to bottom by $\br=(r_1,\dots, r_m)$ and block columns labeled from left to right by $\bc = (c_1,\dots,c_n)$, define $M_{r_i, c_j}$ to be the intersection of block row $r_i$ with block column $c_j$. 
Define $[c_i,c_j]$, $i\leq j$, to be the list $c_i,c_{i+1},\dots, c_j$, and define $M_{[c_i,c_j]}$ to be the submatrix of $M$ consisting of block columns in the list $[c_i,c_j]$. 
Let $N$ be a block matrix with block columns labeled from left to right by $d_1,\dots, d_r, c_{i+1},\dots, c_n$ (i.e. the last $n-i$ column labels of $N$ agree with the last $n-i$ column labels of $M$) such that for each $i+1\leq l\leq n$, the width of block column $c_l$ of $N$ is the same as the width of block column $c_l$ of $M$. 
Given $i<j\leq n$, define
\begin{equation}\label{eq:splitMatrix}
(M,N)_{[[c_i,c_j]]} := \begin{bmatrix}M_{[c_1,c_{i}]}&0&M_{[c_{i+1},c_j]} \\ 0 &N_{[d_1,d_r]}&N_{[c_{i+1},c_j]} \end{bmatrix}.
\end{equation}
Each 0 in \eqref{eq:splitMatrix} denotes a matrix of zeros, with sizes determined by context.

\subsection{Matrices relevant to type $D$ quivers}\label{sect:zigzagDefs}
Fix $n$ and let $Q^*:=Q^*(n)$ be as in \eqref{eq:Dnquiver}. 
Define two $Q^*$-matrices 
as follows, where unlabeled entries are zeros
\begin{equation}\label{eq:M}
A :=
 \begin{blockarray}{ccccccccccc}
          & x_0 & x_1 & x_2 & \cdots & x_{n-1} & x_n\\
      \begin{block}{c[cccccccccc]}
y_0 &{\beta_0}&{\alpha_1} & & &&& \\
y_1 && {\beta_1}&{\alpha_2} & &&&\\
y_2 && & {\beta_{2}} & {\alpha_3}&&\\
\vdots && &&\ddots&  \ddots & \\
y_{n-1} &&&&&{\beta_{n-1}}&{\alpha_n}\\
y_n &&&&&&{\beta_n}\\
\end{block}
\end{blockarray}
\qquad
B :=
 \begin{blockarray}{cccccccccc}
          & x_1 & x_2 & \cdots & x_{n-1} & x_n\\
      \begin{block}{c[ccccccccc]}
y_0' &{\alpha_1'} & & &&& \\
y_1 & {\beta_1}&{\alpha_2} & &&&\\
y_2 & & {\beta_{2}} & {\alpha_3}&&\\
\vdots & &&\ddots&  \ddots & \\
y_{n-1} &&&&{\beta_{n-1}}&{\alpha_n}\\
y_n &&&&&{\beta_n}\\
\end{block}
\end{blockarray}.
\end{equation}

Define a partial order on $Q^*_1$ by defining the following two maximal chains
\begin{equation}
 \beta_0<\alpha_1< \beta_1<\dots<\alpha_n<\beta_n, \quad\text{and}\quad \alpha_1'< \beta_1<\dots<\alpha_n<\beta_n. 
\end{equation}

Given $\gamma, \delta\in Q^*_1$ satisfying $\gamma\leq \delta$ in the above partial order, define the associated \emph{zig-zag matrix} denoted by $[\gamma,\delta]$ to be the rectangular submatrix extracted from $A$  or $B$ whose upper left entry is $\gamma$ and lower right entry is $\delta$.
For example, 
\begin{equation}
[\beta_1,\alpha_3]=\begin{bmatrix}{\beta_1}&{\alpha_2}&0\\0&{\beta_2}&{\alpha_3}\end{bmatrix}, \qquad 
[\alpha_1', \beta_2] = \begin{bmatrix}{\alpha_1'}&0\\ {\beta_1}&{\alpha_2}\\ 0&{\beta_2}\end{bmatrix}, 
\end{equation}
while $[\beta_0, \alpha'_1]$ and $[\alpha_4, \beta_1]$ are not defined. 

Building on the notation in \eqref{eq:splitMatrix}, we define a matrix denoted $\llbracket\gamma,\delta\rrbracket$ for all $\alpha_1\leq \gamma\leq \delta$.
Looking at the examples in Figure \ref{fig:matrixexamples} first will help orient the reader to the general notation.
To begin, we set 
\begin{equation}\label{eq:alpha1def}
\llbracket \alpha_1,\alpha_1\rrbracket := \begin{bmatrix}\beta_0&\alpha_1\\0&\alpha_1'\end{bmatrix} \enskip \text{ and }\enskip \llbracket\alpha_1,\delta \rrbracket:= \begin{bmatrix} [\beta_0,\delta]_{[x_0,x_0]} &[\alpha_1,\delta]\\0&[\alpha_1',\delta] \end{bmatrix}, \delta > \alpha_1,
\end{equation}
where, as usual, the zeros in the above matrices have their sizes determined by context. 
For $i\geq 2$, and $\delta=\alpha_j$ or $\delta=\beta_j$ with $j\geq i$, we set 
\begin{equation}
 \llbracket\alpha_i,\delta\rrbracket:= ([\beta_0,\delta] , [\alpha_1',\delta])_{[[x_{i-1},x_j ]]}.
\end{equation}
For $i\geq 1$ and $\delta=\alpha_{j+1}$ with $j+1 > i$ or $\delta=\beta_j$ with $j \geq i$, we set
\begin{equation}
\llbracket \beta_i,\delta\rrbracket:=  \left( ( [\beta_0,\delta]^T, [\alpha_1',\delta]^T)_{[[y_{i-1},y_j ]]}\right)^T
\end{equation}
where $^T$ denotes the transpose of a matrix.

\begin{figure}
\[
\begin{array}{ll}
\llbracket\alpha_2,\alpha_3\rrbracket = 
\left[\begin{array}{cc|c|cc}
 \beta_0 & \alpha_1 & 0 & 0 & 0\\
 0 & \beta_1 & 0 & \alpha_2 & 0\\
 0 & 0 & 0 & \beta_2 & \alpha_3\\
\hline
 0 & 0 & \alpha'_1 & 0 & 0\\
 0 & 0& \beta_1 & \alpha_2 & 0\\
 0 & 0 & 0 & \beta_2 & \alpha_3\\
\end{array}\right]
\qquad
&\llbracket\beta_1,\alpha_3\rrbracket =
\left[\begin{array}{cccc|ccc}
 \beta_0 & \alpha_1 & 0 & 0 & 0 & 0 & 0\\
\hline
 0 & 0 & 0 & 0 & \alpha'_1 & 0 & 0\\
\hline
 0 & \beta_1 & \alpha_2 & 0 & \beta_1 & \alpha_2 & 0 \\
 0 & 0 & \beta_2 & \alpha_3 & 0 & \beta_2 & \alpha_3 \\
\end{array}\right]
\\
\\
\llbracket\alpha_1,\beta_2\rrbracket = 
\left[\begin{array}{c|cc}
 \beta_0 & \alpha_1  & 0 \\
 0 & \beta_1 & \alpha_2 \\
 0 & 0 & \beta_2 \\
\hline
 0 &  \alpha'_1 & 0 \\
 0 &  \beta_1 & \alpha_2 \\
 0 & 0 & \beta_2 \\
\end{array}\right]
& \llbracket\alpha_1,\beta_2\rrbracket^0 = 
\left[\begin{array}{c|cc}
 0 & \alpha_1  & 0 \\
 0 & \beta_1 & \alpha_2 \\
 0 & 0 & \beta_2 \\
\hline
 0 &  \alpha'_1 & 0 \\
 0 &  \beta_1 & \alpha_2 \\
 0 & 0 & \beta_2 \\
\end{array}\right]
\end{array}
\]
\caption{Examples of matrices $\llbracket\gamma,\delta\rrbracket$ and $\llbracket\gamma,\delta\rrbracket^0$. Lines are added to connect to the block forms appearing in \eqref{eq:splitMatrix} or \eqref{eq:alpha1def}.
\label{fig:matrixexamples}}
\end{figure}

Define $\llbracket\gamma,\delta\rrbracket^{0}$ to be the matrix obtained from $\llbracket\gamma,\delta\rrbracket$ by replacing $\beta_0$ with $0$ (if it appears).
We end the section by defining a distinguished collection of functions $\mathcal{D}_n$ which plays an important role throughout the rest of the paper: to a $Q^*$-matrix $\Phi$, associate the function $\rep_{Q^*}(\bd^*)\rightarrow \mathbb{Z}$, $V\mapsto \text{rank } \Phi_V$. If $\Phi$ is equal to $[\gamma,\delta]$ (respectively $\llbracket \gamma, \delta\rrbracket$, $\llbracket \gamma,\delta\rrbracket^{0}$), denote the associated function by $|\gamma, \delta|$ (respectively, $||\gamma,\delta||$, $||\gamma,\delta||^{0}$). 
Define $\mathcal{D}_n$ to be the set
\begin{equation}\label{eq:distCollection}
\mathcal{D}_n := \left\{|\gamma,\delta| \mid \gamma\leq\delta, \gamma\neq \alpha_1  \right\} \cup \left\{||\gamma,\delta || \mid \alpha_1\leq \gamma\leq \delta \right\}
\cup \left\{|| \alpha_1,\delta ||^{0}\mid \delta\geq \alpha_1\right\}.
\end{equation}

\subsection{Specialization to type $D$ quivers.}
Let $Q$ be a type $D$ quiver with dimension vector $\bd$ and $(Q^*,\bd^*)$ as in Lemma \ref{lem:QtoQn}.  
Recall that $V \in \rep_Q(\bd)$ determines $V*1 \in \rep_{Q^*}(\bd^*)$, as defined in Section \ref{sect:bundles}: concretely, $V*1$ is obtained from $V$ by $(V*1)_a = V_a$ for $a \in Q_1$, and $(V*1)_a=1$ for $a \in Q^*_1 \setminus (Q_1\cup \beta_0)$. 

\begin{proposition}\label{prop:quiverranksorbit}
Let $Q$ be a type D quiver and $V, W$ two representations of the same dimension vector $\bd$.  
Then $V \simeq W$ (i.e. $V$ and $W$ lie in the same $GL(\bd)$-orbit) if and only if $f(V*1) = f(W*1)$ for all $f \in \mathcal{D}_n$, where $\mathcal{D}_n$ is the set from \eqref{eq:distCollection}.
\end{proposition}
\begin{proof}
We only sketch the proof here, since it follows from standard computations with projective resolutions of quiver representations, and a more detailed account would require a significant amount of additional \textit{ad hoc} notation used only in this proof.
The key point explained below is that for a given $V$, the values $f(V*1)$ for $f \in \mathcal{D}_n$ 
along with $\bd$
determine the values $\dim \Hom_Q(N,V)$ for all representations $N$ of $Q$.
These values, in turn, determine the isomorphism class of $V$ by a theorem of Auslander \cite{Auslander:1982fk}. That is, two representations $V, W$ of $Q$ satisfy $V \simeq W$ if and only if $\dim \Hom_Q(N,V)=\dim\Hom_Q(N,W)$ for all representations $N$ of $Q$.  
Using linearity of $\Hom$, to prove the proposition it thus suffices to show that for each $N$ indecomposable, $\dim \Hom_Q(N,V)$ can be obtained from the values $f(V*1)$ for $f \in \mathcal{D}_n$ and $\bd$.

Let $P_1 \xto{\phi} P_2$ be a morphism of projective representations of $Q$ such that $\coker \phi \simeq N$.
For any representation $V$ of $Q$, we get an induced exact sequence of vector spaces
\begin{equation}
0 \to \Hom_Q(N,V) \to \Hom_Q(P_2, V) \xto{\Hom_Q(\phi,V)} \Hom_Q(P_1, V).
\end{equation}
If we denote by $P(x)$ the projective cover of the simple representation $S(x)$ concentrated at vertex $x$, then $\dim\Hom_Q(P(x), V)=\bd(x)$.
Thus, the dimensions of $\Hom_Q(P_1, V)$ and $\Hom_Q(P_2, V)$ depend only on $\bd$, and knowing $\dim \Hom_Q(N,V)$ is equivalent to knowing $\rank \Hom_Q(\phi,V)$ and $\bd$.
Choosing a decomposition of $P_1, P_2$ into indecomposables, the morphism $\phi$ is represented by a $Q$-matrix $\overline{\Phi}$, and its evaluation $\overline{\Phi}_V$ represents $\Hom_Q(\phi,V)$.
So it remains to show that, for a given $V$, the values $f(V*1)$ for $f \in \mathcal{D}_n$ and $\bd$ determine
the ranks of the matrices $\overline{\Phi}_V$,
 as $\phi$ varies over a collection of morphisms whose collection of cokernels contains every indecomposable representation of $Q$ (up to isomorphism and projective summands).

Given any $Q^*$-matrix $\Phi$, we get a $Q$-matrix $\overline{\Phi}$ as follows.
First apply the map $\nu$ to the vertices labeling its rows and columns (see Section \ref{sect:contract}), and when $Q$ is oriented as in \eqref{eq:fundorient2}, delete any columns labeled by $x_0$ (this doesn't affect the values of the functions $f(V*1)$ since $V*1$ has dimension 0 at $x_0$ anyways).  Then replace each arrow in this matrix from the set $Q^*_1 \setminus Q_1$ with a 1 to get $\overline{\Phi}$.

Now consider the set of $Q^*$-matrices
\begin{equation}
\mathbf{D}_n:= \left\{[\gamma,\delta] \mid \gamma\leq\delta, \gamma\neq \alpha_1  \right\} \cup \left\{\llbracket\gamma,\delta \rrbracket \mid \alpha_1\leq \gamma\leq \delta \right\} \cup \left\{\llbracket \alpha_1,\delta \rrbracket^{0}\mid \delta\geq \alpha_1\right\},
\end{equation}
that is, the set of $Q^*$-matrices used to obtain the set $\cD_n$.
It is straightforward to check that every isomorphism class of indecomposable representation of $Q$, up to a projective summand, appears as the cokernel of some $\overline{\Phi}$ with $\Phi \in \mathbf{D}_n$.  (Many such specializations will yield decomposable representations, and they are not necessarily distinct; we only claim that all indecomposable representations appear, possibly with additional projective summands, amongst others.) 

Now it follows directly from the definitions that if $f \in \mathcal{D}_n$ and $\Phi \in \mathbf{D}_n$ is the corresponding $Q^*$-matrix, then $f(V*1)=\rank \overline{\Phi}_V$.  So the values $f(V*1)$ for $f \in \mathcal{D}_n$ determine the isomorphism class of $V$ among representations of the same dimension vector.
\end{proof}

\section{A linear slice of $K \backslash G$}\label{sect:slice}
Let $G = GL(a+b)$ and $K = GL(a)\times GL(b)$ embedded as two blocks along the diagonal of $G$.
Let $B_+$, $B_-$, $U_+$, $U_-$, and $T$ 
respectively denote the Borel subgroup of upper triangular matrices in $G$, the Borel subgroup of lower triangular matrices in $G$, the unipotent radical of $B_+$, the unipotent radical of $B_-$, and the maximal torus of diagonal matrices in $G$. 

The goal of this section is to transfer the problems we study for orbit closures of certain parabolic subgroups in $K\backslash G$ to intersections of these orbit closures with a locally closed subvariety, or \emph{slice}, in $K \backslash G$.
This is accomplished by Theorem \ref{thm:intersectS}, which is an analogue of a result on orbit closures in flag varieties stated by Kazhdan and Lusztig \cite[Lem.~A.4]{KL79}.

\begin{remark}
Mars and Springer have a related general result about $B_+$-orbit closures in $K\backslash G$ \cite[\S6.4]{MarsSpringer}. This was recently used by Woo, Wyser, and Yong to study singularities of these orbit closures \cite{WooWyserYong}.  
Our Theorem \ref{thm:intersectS} is not a special case of the Mars-Springer result; our $S$ (see Section \ref{sect:linearSliceDefinition}) is not a Mars-Springer slice, and our result applies only to certain parabolic group orbit closures rather than to all $B_+$-orbit closures. 
Our slice is more natural for our quiver application: it allows for a type $D$ quiver result which parallels the type $A$ result in \cite{KR}, and does not require exclusion of characteristic 2.
\end{remark}

\subsection{A slice in $K\backslash G$}\label{sect:linearSliceDefinition}

Fix positive integers $q,r,s,t$ such that $a = q+r+s$ and $b = t+r+s$ (these will be determined by a quiver dimension vector $\bd$ in the next section). Let $v_0, w$ be the block matrices below and let $B'$ be the subgroup of $GL(a+b)$ consisting of all matrices of the form shown below:

\[
v_0=
\begin{blockarray}{ccccccccccc}
          & s & s & r & q & t & r \\
      \begin{block}{c[cccccccccc]}
q & \cdot & \cdot & \cdot & J & \cdot & \cdot \\
r & \cdot & \cdot & J & \cdot & \cdot & \cdot \\
s & J & \cdot & \cdot & \cdot & \cdot & \cdot \\
t & \cdot & \cdot & \cdot & \cdot & 1 & \cdot \\
r & \cdot & \cdot & J & \cdot & \cdot & 1 \\
s & J & 1 & \cdot & \cdot & \cdot & \cdot \\
      \end{block}
    \end{blockarray},
    \enskip
  w  =
\begin{blockarray}{ccccccccccc}
        & s & s &r & q & t & r   \\
      \begin{block}{c[cccccccccc]}
s&\cdot & J & \cdot & \cdot &\cdot &\cdot \\
r&\cdot & \cdot & \cdot & \cdot &\cdot  & J \\
q&\cdot & \cdot & \cdot & J  & \cdot & \cdot \\
t&\cdot & \cdot & \cdot  & \cdot & 1 & \cdot \\
r&\cdot & \cdot & J & \cdot  &\cdot & \cdot \\
s&J & \cdot & \cdot & \cdot  &\cdot & \cdot \\
\end{block}
\end{blockarray},
\enskip
B': \begin{blockarray}{ccccccccccc}
          & s & r & q & t & r & s \\
      \begin{block}{c[cccccccccc]}
s & \ell_1 & \cdot & \cdot & \cdot & \cdot & \cdot \\
r & * & \ell_2 & \cdot & \cdot & \cdot & \cdot \\
q & * & * & 1 & \cdot & \cdot & \cdot \\
t & * & * & \cdot & 1 & \cdot & \cdot \\
r & * & * & * &* & u_1 & \cdot \\
s & * & * & * & * & * & u_2 \\
      \end{block} 
    \end{blockarray}
\]
The labels at the top and left of the matrices indicate sizes of block rows and block columns. Dots denote blocks of zeros. Each $1$ is an identity matrix, $J$ is a square matrix with ones along the antidiagonal and zeros elsewhere, each asterisk is an arbitrary matrix, $\ell_i$ is invertible lower triangular, and $u_i$ is unipotent lower triangular.

Let $U'$ denote the unipotent radical of $B'$ and let $\overline{v}_0:=Kv_0 \in K\backslash G$. We define the \emph{slice} $S$ to be the orbit of $\overline{v}_0$ under the action of the group $w^{-1} U' w\cap U_-$ by right multiplication:
\begin{equation}
S:= \overline{v}_0 (w^{-1} U' w\cap U_-).
\end{equation}

We will typically represent the slice by an isomorphic space of matrices in $G$. To this end, let $\tilde{S}\subseteq G$ be the space of matrices of the form
\begin{equation}\label{eq:v01}
\tilde{S}:\qquad\begin{blockarray}{ccccccccccc}
          & s & s & r & q & t & r \\
      \begin{block}{c[cccccccccc]}
q & \cdot & c & \cdot & J & \cdot & \cdot \\
r & \cdot & d & J & \cdot & \cdot & \cdot \\
s & J & \cdot & \cdot & \cdot & \cdot & \cdot \\
t & \cdot & e & \cdot & \cdot & 1 & \cdot \\
r & \cdot & f & J & \cdot & \cdot & 1 \\
s & J & 1 & \cdot & \cdot & \cdot & \cdot \\
      \end{block}
    \end{blockarray}
\end{equation}
where $c,d,e,f$ are arbitrary matrices over $k$ of indicated sizes.

\begin{lemma}\label{lem:sliceIso}
The following map is an isomorphism:
\begin{equation}\label{eq:Bprimeconjaction}
w^{-1}U'w\cap U_- \to \tilde{S}, \qquad u  \mapsto v_0 u .
\end{equation}
\end{lemma}
\begin{proof}

Consider the spaces of matrices of the forms
\begin{equation}\label{eq:matrixComps}
  \begin{blockarray}{ccccccccccc}
          & s & s & r & q &t & r \\
      \begin{block}{c[cccccccccc]}
s & 1 & \cdot & \cdot & \cdot & \cdot & \cdot \\
s & \cdot & 1 & \cdot & \cdot & \cdot & \cdot \\
r & \cdot & c' & 1 & \cdot & \cdot & \cdot \\
q & \cdot & d' & \cdot & 1 & \cdot & \cdot \\
t & \cdot & e' & \cdot & \cdot &1 & \cdot \\
r & \cdot & f' & \cdot & \cdot & \cdot&1 \\
      \end{block}
    \end{blockarray},
\quad \text{and} \quad
   \begin{blockarray}{ccccccccccc}
          & s & s & r & q &t & r \\
      \begin{block}{c[cccccccccc]}
q & \cdot & Jd' & \cdot & J & \cdot & \cdot \\
r & \cdot & Jc' & J & \cdot & \cdot & \cdot \\
s & J & \cdot & \cdot & \cdot & \cdot & \cdot \\
t & \cdot &e'& \cdot& \cdot &1&\cdot \\
r & \cdot &f'+Jc' & J & \cdot & \cdot & 1 \\
s & J & 1 & \cdot & \cdot & \cdot &\cdot \\
      \end{block}
    \end{blockarray}
\end{equation}
where $c',d',e',f'$ are arbitrary matrices of indicated sizes.
One easily checks that $w^{-1}U'w\cap U_-$ is the affine space of matrices on the left  
in \eqref{eq:matrixComps} and furthermore, if $u$ is a matrix on the left then $v_0u$ is the corresponding matrix on the right. Then a linear change of variables identifies the space of matrices of the form on the right with the space of matrices in \eqref{eq:v01}.
\end{proof}

\begin{proposition}\label{prop:slicesIsom}
The projection map $\pi: \tilde{S}\rightarrow S$, $g\mapsto Kg$, is an isomorphism of varieties.
\end{proposition}

\begin{proof}
Let $V$ denote the space of matrices of the following form:
\begin{equation}\label{eq:V}
V:~~\begin{blockarray}{ccccccccccc}
          & s & s & r & q & t & r \\
      \begin{block}{c[cccccccccc]}
q & \cdot & * & \cdot & J & * & * \\
r & \cdot & * & J & \cdot & * & * \\
s & J & * & \cdot & \cdot & * & * \\
\cline{1-7}
t & \cdot & * & * & * & 1 & \cdot \\
r & \cdot & * & * & * & \cdot & 1 \\
s & J & * & * & * & \cdot & \cdot \\
      \end{block}
    \end{blockarray}\subseteq \text{Mat}(a+b,a+b)
\end{equation}
Observe that the map $\pi: V\rightarrow K\backslash V$, $v\mapsto Kv$ is an isomorphism. Indeed, identify $V$ with a product of matrix spaces $V_1\times V_2$, where $V_1\subseteq \text{Mat}(a,a+b)$ is the space of matrices above the dividing line of \eqref{eq:V} and $V_2\subseteq  \text{Mat}(b,a+b)$ is the space of matrices below the dividing line. Then observe that the map $\pi$ is the same as the isomorphism $V_1\times V_2\rightarrow GL(a)\backslash V_1\times GL(b)\backslash V_2\subseteq \Gr(a,a+b)\times \Gr(b,a+b)$ which identifies $V_1\times V_2$ with a product of standard open sets in the product of Grassmannians $\Gr(a,a+b)\times \Gr(b,a+b)$. 
By Lemma \ref{lem:sliceIso}, we have that $\tilde{S} = v_0(w^{-1}U'w\cap U_-)\subseteq V$. Thus, $\tilde{S}$ maps isomorphically onto $S$ by the restriction of $\pi$ to $\tilde{S}$.
\end{proof}

Finally, let $T_S\leq G$ be the set of matrices $t$ of the form below, where $t_1, \dotsc, t_4$ are invertible diagonal matrices, and consider its action on $S$ by right multiplication.  For $M \in \tilde{S}$ as in \eqref{eq:v01} and $t \in T_S$ as below, a representative of $M \cdot t$ in $\tilde{S}$ is given by
\begin{equation}\label{eq:TSdef}
t= \begin{blockarray}{ccccccccccc}
          & s & s & r & q &t & r \\
      \begin{block}{c[cccccccccc]}
s & Jt_1J & \cdot & \cdot & \cdot & \cdot & \cdot \\
s & \cdot & t_1 & \cdot & \cdot & \cdot & \cdot \\
r & \cdot & \cdot & Jt_2J & \cdot &\cdot & \cdot \\
q & \cdot & \cdot & \cdot & J t_3 J & \cdot & \cdot \\
t & \cdot & \cdot & \cdot & \cdot &t_4 & \cdot \\
r & \cdot & \cdot & \cdot & \cdot & \cdot &t_2 \\
      \end{block}
    \end{blockarray}
    \qquad \qquad
M \cdot t = 
\begin{blockarray}{ccccccccccc}
          & s & s & r & q & t & r \\
      \begin{block}{c[cccccccccc]}
q & \cdot & t_3^{-1} ct_1 & \cdot & J & \cdot & \cdot \\
r & \cdot & t^{-1}_2 d t_1 & J & \cdot & \cdot & \cdot \\
s & J & \cdot & \cdot & \cdot & \cdot & \cdot \\
t & \cdot & t_4^{-1}et_1 & \cdot & \cdot & 1 & \cdot \\
r & \cdot & t^{-1}_2 f t_1 & J & \cdot & \cdot & 1 \\
s & J & 1 & \cdot & \cdot & \cdot & \cdot \\
      \end{block}
    \end{blockarray}.
\end{equation}

\subsection{Proof of Theorem \ref{thm:intersectS}}\label{sect:linearSlice}
To begin, we identify an open subvariety of $K\backslash G$ with a product of subgroups of $G$. 
We retain the notation of the previous section, 
but coarsen the $6\times 6$ block matrix structure to a $4\times 4$ block matrix structure whenever possible. Indeed, let $p:=s+r$. Then matrices in $B'$ take the form indicated below. Furthermore, 
let $o = v_0w^{-1}$, and let $Q$ be the subgroup of $GL(a+b)\times GL(a+b)$ consisting of pairs of the form indicated below: 
\begin{equation}\label{eq:subgroup1}
o=
\begin{blockarray}{cccccccc}
           &p& q & t& p \\
      \begin{block}{c[ccccccc]}
q & \cdot & 1 & \cdot & \cdot  \\
p & \cdot & \cdot & \cdot & 1 \\
t & \cdot & \cdot & 1  & \cdot \\
p & J & \cdot & \cdot & 1 \\
      \end{block}
    \end{blockarray},
        \quad
    B': ~~ \begin{blockarray}{cccccccc}
           &p& q & t& p \\
      \begin{block}{c[ccccccc]}
p & \ell & \cdot & \cdot & \cdot  \\
q & * & 1 & \cdot & \cdot \\
t & * & \cdot & 1  & \cdot \\
p & * & * & * & u \\
      \end{block}
    \end{blockarray},
    \quad
    Q: ~~\left( \begin{blockarray}{cccccccc}
           &p& q & t& p \\
      \begin{block}{c[ccccccc]}
p & 1 & \cdot & \cdot & \cdot  \\
q & \cdot & 1 & y & \cdot \\
t & \cdot & \cdot & 1  & \cdot \\
p & \cdot & \cdot & \cdot & 1 \\
      \end{block}
    \end{blockarray},
     \begin{blockarray}{cccccccc}
           &p& q & t& p \\
      \begin{block}{c[ccccccc]}
p & 1 & \cdot & \cdot & \cdot  \\
q & \cdot & 1 & \cdot & \cdot \\
t & \cdot & z & 1  & \cdot \\
p & \cdot & \cdot & \cdot & 1 \\
      \end{block}
    \end{blockarray}\right).
\end{equation}
The asterisks, $ y, $ and $z$ are arbitrary matrices, $\ell$ is an invertible lower triangular matrix, and $u$ is a lower triangular matrix with $1$s along the diagonal. Write $Q = Q_1\times Q_2$ so that elements in $Q_1$ are upper triangular and elements in $Q_2$ are lower triangular. 

 \begin{proposition}\label{lem:openImmersion1}
The following morphism is an open immersion.
\[
\mu: oB'o^{-1} \times oQo^{-1} \to K\backslash G,\qquad (obo^{-1}, (oq_1o^{-1}, oq_2o^{-1})) \mapsto K o bq_1q_2 o^{-1}
\]
 \end{proposition}
 \begin{proof}
It suffices to show that $\mu$ is \'{e}tale and universally injective \cite[Lemma 34.22.2]{stacks-project}, which, because our base field is algebraically closed, is equivalent to showing that $\mu$ is \'{e}tale and injective on points.
We start with injectivity.  
Suppose $Kobq_1q_2o^{-1} = Kob'q_1'q_2'o^{-1}$ for some $(b,(q_1,q_2)), (b',(q_1',q_2'))\in B'\times Q$. 
Then $obq_1q_2q_2'^{-1}q_1'^{-1} b'^{-1}o^{-1} \in K$. 
We also have that this element is in $oB'Q_1Q_2Q_1B'o^{-1}$. 
We will show that $ oB'Q_1Q_2Q_1B'o^{-1} \cap K=1$. This implies injectivity of $\mu$: if we have $obq_1q_2q_2'^{-1}q_1'^{-1} b'^{-1}o^{-1} =1$ then $q_1q_2q_2'^{-1}q_1'^{-1} = b^{-1}b'$, and so $b = b'$ and $q_1q_2q_2'^{-1}q_1'^{-1}=1$ since $B'\cap (Q_1Q_2Q_1) = 1$. Then, because $Q_1\cap Q_2 = 1$, the equality $q_1q_2q_2'^{-1}q_1'^{-1}=1$ implies that $q_1 = q_1'$ and $q_2=  q_2'$.

  So, consider an element $obq_1q_2q_1'b'o^{-1}\in oB'Q_1Q_2Q_1B'o^{-1}$. With the help of Macaulay2 \cite{M2}, one can perform a direct matrix computation to observe that if this element is in $K$, then $q_2=1$. 
  Thus, $oB'Q_1Q_2Q_1B'o^{-1}\cap K  = oB'Q_1B'o^{-1}\cap K$. 
A similar computation shows $oB'Q_1B'o^{-1}\cap K = oB'o^{-1}\cap K$, and one more matrix computation shows that $oB'o^{-1}\cap K = 1$. In the last computation, it is important to remember that $\ell$ and $u$ (as in \eqref{eq:subgroup1}) are lower triangular and unipotent lower triangular respectively, and to note that $J\ell J$ is the $180^\circ$-rotation of $\ell$. 
Thus, $ oB'Q_1Q_2Q_1B'o^{-1} \cap K=1$ and so $\mu$ is injective.

We next show that $\im(\mu)$ is \'{e}tale by showing that it is smooth of relative dimension $0$. 
We note that $oB'o^{-1}\times oQ_1o^{-1}\times oQ_2o^{-1}$ and $K\backslash G$ are smooth varieties of the same dimension, so it suffices to show that $\mu$ is submersive.
To begin, note that an arbitrary element of $T_{(1,1,1)}(oB'o^{-1}\times oQ_1o^{-1}\times oQ_2o^{-1})\cong T_1(oB'o^{-1})\times T_1(oQ_1o^{-1})\times T_1(oQ_2o^{-1})$ has the form
\begin{equation}\label{eq:tangent}
\left( \begin{blockarray}{cccccccc}
           &q& p & t& p \\
      \begin{block}{c[cc|ccccc]}
q & \cdot & -X_1J & \cdot & X_1J  \\
p & X_5 & -X_4J+U & X_6 & X_4J \\
\cline{1-7}
t & \cdot & -X_2J & \cdot  & X_2J \\
p & X_5 & -JLJ-X_4J+U & X_6 & JLJ+X_4J \\
      \end{block}
    \end{blockarray} ,
\begin{blockarray}{cccccccc}
           &q& p & t& p \\
      \begin{block}{c[cc|ccccc]}
q & \cdot & \cdot & Y & \cdot  \\
p & \cdot & \cdot & \cdot & \cdot \\
\cline{1-7}
t & \cdot & \cdot & \cdot  & \cdot \\
p & \cdot & \cdot & \cdot & \cdot \\
      \end{block}
    \end{blockarray},
 \begin{blockarray}{cccccccc}
           &q& p & t& p \\
      \begin{block}{c[cc|ccccc]}
q & \cdot & \cdot & \cdot & \cdot  \\
p & \cdot & \cdot & \cdot & \cdot \\
\cline{1-7}
t & Z & \cdot & \cdot  & \cdot \\
p & \cdot & \cdot & \cdot & \cdot \\
      \end{block}
    \end{blockarray}
\right)
\end{equation}
where $X_i, Y, Z$ are arbitrary matrices of the indicated sizes, $U$ is a strictly lower triangular matrix, and $L$ is a lower triangular matrix.

Let $m: G\times G\times G\rightarrow G$ denote the multiplication map $(g_1,g_2,g_3)\mapsto g_1g_2g_3$. 
This induces a map on the tangent spaces at the identity $dm_{(1,1,1)}:T_{(1,1,1)}(G\times G\times G)\rightarrow T_{1}G$ given by addition (i.e. $(\alpha,\beta,\gamma) \mapsto \alpha+\beta+\gamma$). 
Let $i: oB'o^{-1}\times oQ_1o^{-1}\times oQ_2o^{-1}\rightarrow G\times G\times G$ be inclusion and let $\pi: G\rightarrow K\backslash G$ be the map which sends an element $g$ to its coset in $K\backslash G$. Then $\mu = \pi\circ m\circ i$, and so the map $d\mu_{(1,1,1)}: T_{(1,1,1)}(oB'o^{-1}\times oQ_1o^{-1}\times oQ_2o^{-1})\rightarrow T_{K}(K\backslash G)\cong T_1(K)\backslash T_1(G)$ is also given by addition. Thus, the kernel of $d\mu_{(1,1,1)}$ consists of those triples appearing in \eqref{eq:tangent} which sum to an element in $T_1(K)$. This only happens when $X_i, Y, Z, U, L$ are all zero. Thus, the kernel of $d\mu_{(1,1,1)}$ is trivial and $d\mu_{(1,1,1)}$ is injective. Since the dimensions of the domain and codomain of $d\mu_{(1,1,1)}$ agree, $d\mu_{(1,1,1)}$ is an isomorphism. 
Finally, fix $(b,q_1,q_2)\in B'\times Q_1\times Q_2$ and observe that $\mu$ can be factored as 
\begin{equation}
oB'o^{-1}\times oQ_1o^{-1}\times oQ_2o^{-1}\xrightarrow{t_1} oB'o^{-1}\times oQ_1o^{-1}\times oQ_2o^{-1}\xrightarrow{\mu} K\backslash G \xrightarrow{t_2} K\backslash G
\end{equation}
where $t_1$ and $t_2$ are right translations by $(ob^{-1}o^{-1},oq_1^{-1}o^{-1},oq_2^{-1}o^{-1})\in oB'o^{-1}\times oQ_1o^{-1}\times oQ_2o^{-1}$ and $obq_1q_2o^{-1}$ respectively. Then,  
\begin{equation}
(obo^{-1},oq_1o^{-1},oq_2o^{-1})\xmapsto{t_1} (1,1,1) \xmapsto{\mu} K1 \xmapsto{t_2} Kobq_1q_2o^{-1}
\end{equation}
and so $d\mu_{(obo^{-1},oq_1o^{-1},oq_2o^{-1})}$ is an isomorphism since $d\mu_{(1,1,1)}$ is an isomorphism as are the tangent maps induced by  the right translations. 
This completes the proof. 
\end{proof}

\begin{remark}\label{rmk:GBvsGK}
In the classical $G/B$ situation, the action of the opposite Borel restricts to a $T$-equivariant isomorphism from its unipotent radical to an open orbit in $G/B$.  This greatly simplifies the analogue of our Theorem \ref{thm:intersectS} in that situation.
In our setup, we were able to find a subgroup of $G$ mapping isomorphically onto an open subset of $G/K$, but 
it did not give rise to a \emph{torus-equivariant} isomorphism in Theorem \ref{thm:intersectS}.
\end{remark}

We now proceed to factorize $B'$, first introducing some useful notation.
Let $\Phi^-, \Phi^+$ be the sets of negative and positive roots of the root system of $G$ (relative to $T, B_+$). We identify $\Phi^- =\{ (i,j) \mid 1 \leq j < i \leq a+b\}$ and $\Phi^+ =\{ (i,j) \mid 1 \leq i < j \leq a+b\}$.
For $\alpha = (i,j) \in \Phi^- \cup \Phi^+$, let 
$u_\alpha\colon (k, +) \to G$ be the morphism of algebraic groups sending $c \in k$ to the matrix which is equal to the identity matrix except that the $(i,j)$ entry is $c$.  Then $u_\alpha$ is an isomorphism onto its image, known as a 1-parameter subgroup, which we denote by $U_\alpha$.

Let $I':=\{\alpha \in \Phi^- \mid U_\alpha \leq U'\}$, and consider the partition $I'=I'_+ \coprod I'_-$ where
\begin{equation}
I'_\pm := \setst{\alpha \in I'}{w\cdot \alpha \in \Phi^\pm}.
\end{equation}
Define subgroups of $U'$ as products of 1-parameter subgroups
\begin{equation}\label{eq:Upm}
U'_{\pm,w}=\prod_{\alpha \in I'_\pm} U_\alpha.
\end{equation}
Note that $U'_{-,w} = U'\cap wU_-w^{-1}$ and $U'_{+,w} = U'\cap wU_+w^{-1}$.

\begin{lemma}\label{lem:Bprimedecomp}
Let $U'\leq B'$ be the unipotent radical and $T' \leq B'$ the maximal torus of diagonal matrices in $B'$.  Then the product map  
\begin{equation}\label{eq:Bprimeiso}
U'_{-,w} \times U'_{+,w} \times T' \to B'
\end{equation}
is an isomorphism of varieties.
\end{lemma}
\begin{proof}
By \cite[Thm.~6.3.5(iv)]{Springerbook}, the product map $U' \times T' \to B'$ is an isomorphism of varieties.  
Therefore it is enough to show that the product map
\begin{equation}\label{eq:productisoclaim}
U'_{-,w} \times U'_{+,w} \to U'
\end{equation}
is also an isomorphism of varieties.

The hypotheses of \cite[Lem.~8.2.2]{Springerbook} hold for the connected, solvable linear algebraic group\footnote{We need $T$ instead of $T'$ here for condition (a) of the cited lemma to hold.} 
$U'T$ and set of isomorphisms $\{u_\alpha \mid \alpha \in I'\}$.  Therefore, the multiplication map
\begin{equation}\label{eq:Uprimeisomorphism}
\prod_{\alpha \in I'} u_\alpha \colon \prod_{\alpha \in I'} (k, +) \to U', \qquad (c_\alpha)_{\alpha \in I'} \mapsto \prod_{\alpha \in I'} u_\alpha(c_\alpha)
\end{equation}
is an isomorphism, where the product can be taken in any order.
Therefore, we choose the order where factors indexed by $I'_-$ come before those of $I'_+$ to rewrite \eqref{eq:Uprimeisomorphism} in the form \eqref{eq:productisoclaim}.
\end{proof}

Let $s,r,q,t$ be as in the previous subsection, and let $P'\leq G$ be the parabolic subgroup of block upper triangular matrices where the first $2s+r$ blocks, listed from top left to bottom right, have size $1$, the next block has size $q+t$, and the last $r$ blocks have size $1$.  

Left conjugating each factor of \eqref{eq:Bprimeiso} by $o$ and applying this isomorphism to the one in Proposition \ref{lem:openImmersion1}, we obtain an open immersion to a neighborhood of $1\,K \in K \backslash G$:
\begin{equation}
oU'_{-,w}o^{-1} \times oU'_{+,w}o^{-1} \times oT'o^{-1} \times oQo^{-1} \to K\backslash G.
\end{equation}
Then right conjugating each factor by $ow$ and using the isomorphisms of Lemma \ref{lem:sliceIso} and Proposition \ref{prop:slicesIsom} , we obtain an open immersion which can be written as 
\begin{equation}\label{eq:Psimap}
\Psi\colon S \times Z \rightarrow K\backslash G, \qquad (s,z) \mapsto s\, p(z)
\end{equation}
where $Z:=\left(w^{-1}U'_{+,w}w\right)\times\left(w^{-1}T'w\right)\times\left(w^{-1}Qw\right)$ and $p\colon Z\to P'$ is the multiplication map.  It can be directly checked that $w^{-1}Q_iw\leq P'$, $i = 1,2$, and since $w$ normalizes $T$, the image of $p$ is indeed in $P'$.

We will now see that this morphism is equivariant with respect to the torus $T_S\leq G$ defined in \eqref{eq:TSdef}. 
We let $T_S$ act by right multiplication on $S$ and on each factor of $Z$ by conjugation.  
The factors of $Z$ are stable under conjugation by the larger group $T$ since $w$ is a permutation matrix, and thus are also stable under conjugation by $T_S$.  It is immediate by cancelation of inverses that $\Psi$ is equivariant with respect to this action on the left hand side, and right multiplication of $T_S$ on $K\backslash G$.

\begin{theorem}\label{thm:intersectS}
For any $P'$-orbit closure $\overline{O} \subseteq K\backslash G$, the morphism $\Psi$ of \eqref{eq:Psimap} restricts to a $T_S$-equivariant isomorphism of schemes, where $\im \Psi \subset K \backslash G$ is considered as an open subscheme:
\begin{equation}\label{eq:Psiproperty}
(\overline{O} \cap S)\times Z \simeq \overline{O} \cap \im \Psi .
\end{equation}
Furthermore, the statement above remains true when $P'$-orbits are replaced by $P$-orbits, where $P \leq G$ is any parabolic subgroup containing $P'$.
\end{theorem}
\begin{proof}
We may assume that $\overline{O} \cap \im \Psi$ is nonempty, otherwise the statement is trivial.
We first see that \eqref{eq:Psiproperty} holds at the level of varieties (i.e. taking the induced reduced subscheme of $\overline{O} \cap S$).  The restriction of the open immersion $\Psi$ to the closed subvariety $(\overline{O} \cap S)\times Z \subset S \times Z$ is an isomorphism of varieties onto its image, which is contained in $\overline{O} \cap \im \Psi$ since $\overline{O}$ is $P'$-stable. 
To see that the image of this restriction is exactly $\overline{O} \cap \im \Psi$, simply note that if $\Psi(s,z)=s\,p(z) \in \overline{O} \cap \im \Psi$, then $s=(s\,p(z))p(z)^{-1} \in \overline{O}$ as well.  So we actually have $(s,z) \in (\overline{O} \cap S)\times Z$, and thus the image of the restriction is exactly $\overline{O} \cap \im \Psi$.

Now let $W$ be the scheme-theoretic image of the restriction of $\Psi$ to the closed subscheme $(\overline{O} \cap S)\times Z$ (both of which could be nonreduced, a priori). 
Then $\Psi$ restricts to an open immersion from $(\overline{O} \cap S)\times Z$ to $W$.  
From the first paragraph, we know that the underlying topological spaces of $W$ and $\overline{O}$ are the same (each is the closure of $\Psi((\overline{O} \cap S)\times Z)=\overline{O} \cap \im \Psi$ in $K\backslash G$).
Since $(\overline{O} \cap S) \subset \overline{O}$ and the latter is $P'$-stable, we also have that $W \subseteq \overline{O}$ is a closed subscheme.  But the orbit closure $\overline{O}$ is reduced, so  $W=\overline{O}$ as subschemes of $K\backslash G$.
\end{proof}

\section{Embedding representation spaces of $Q^*(n)$ into a slice}\label{sect:embedding}
Let $Q^*(n)$ be as in \eqref{eq:Dnquiver}, and let $\bd^*$ be a dimension vector for $Q^*(n)$. 
To the representation space $\rep_{Q^*(n)}(\bd^*)$ we associate a symmetric variety $K\backslash G$ and a slice $S(\bd^*)\subseteq K\backslash G$ as in Section \ref{sect:linearSliceDefinition}. We also define a closed embedding $\eta:\rep_{Q^*(n)}\rightarrow S(\bd^*)$, which will be an important ingredient in our proof of Theorem \ref{thm:mainTheorem}. 

\subsection{A particular slice and closed embedding}\label{sect:setupNotation}
Let $Q^*(n)$ be as in \eqref{eq:Dnquiver}, and let $\bd^*$ be a dimension vector for $Q^*(n)$. 
We denote the total dimension by $d^*$, and let $a = d^*-\bd^*(y_0')$ and $b = d^*-\bd^*(y_0)$. Let $G = GL(a+b)$, and let $K=GL(a) \times GL(b)$ embedded as two blocks along the diagonal of $G$. 
Let $d^*_x = \sum_{i=0}^n\bd^*(x_i)$ and $d^*_y = \sum_{i=1}^n\bd^*(y_i)$.

Let $S(\bd^*)\subseteq K\backslash G$ denote the image, under the natural projection map $G\rightarrow K\backslash G$, of the space of matrices from \eqref{eq:v01} where we set $q = \bd^*(y_0)$, $r = d^*_y$, $s = d^*_x$, and $t = \bd^*(y_0')$.  Then $S(\bd^*)$ consists of points represented by matrices of the form below, where a block filled with zeroes is indicated by a dot:
\begin{equation}\label{eq:Sd}
 \begin{blockarray}{cccccccccccc}
          & d^*_x & d^*_x & d^*_y & \bd^*(y_0) & \bd^*(y_0') & d^*_y \\
      \begin{block}{c[ccccccccccc]}
\bd^*(y_0) & \cdot & * & \cdot & J &  \cdot & \cdot \\
d^*_y & \cdot & * & J & \cdot & \cdot & \cdot \\
d^*_x & J & \cdot & \cdot & \cdot & \cdot & \cdot \\
\cline{1-11}
\bd^*(y_0') & \cdot & \star & \cdot & \cdot & 1 & \cdot \\
d^*_y & \cdot & \star & J & \cdot & \cdot & 1 \\
d^*_x & J & 1 & \cdot & \cdot & \cdot & \cdot \\
      \end{block}
    \end{blockarray}.
\end{equation}
Note that the horizontal dividing line in the above matrix becomes important in Section \ref{sect:Smirnov}.

Let $A$ and $B$ be the zig-zag matrices from \eqref{eq:M}, 
and let $A_V$ and $B_V$ denote their evaluation at $V\in \rep_{Q^*(n)}(\bd^*)$. 
Let $B'_V$ be the matrix $[0 ~~B_V] $ where $0$ denotes a zero matrix of width $\bd^*(x_0)$. 
Define a closed embedding $\eta\colon \rep_{Q^*(n)}(\bd^*)\rightarrow S(\bd^*)$ by 
\begin{equation}\label{eq:defEta}
\eta(V)= 
\vcenter{\hbox{
  \begin{tikzpicture}
     \matrix (M) [%
       matrix of math nodes, column sep=20, row sep=6,
       nodes in empty cells,
       inner sep=3pt,
       left delimiter={[},
    right delimiter={]}
     ] {%
\cdot & \phantom{A} & \cdot &  J & \cdot & \cdot \\
\cdot & \phantom{A} & J & \cdot & \cdot & \cdot  \\
J & \cdot & \cdot & \cdot & \cdot & \cdot  \\
\cdot & \phantom{A} & \cdot &  \cdot& 1 & \cdot  \\
\cdot & \phantom{A} & J & \cdot & \cdot & 1  \\
J & 1 & \cdot & \cdot & \cdot & \cdot \\
    };
\draw (M-1-2.north west) rectangle (M-2-2.south east) node[pos=.5,scale=1] {$A_V$};
\draw (M-4-2.north west) rectangle (M-5-2.south east) node[pos=.5,scale=1] {$B'_V$};
     \node[anchor=south east, left=25pt,above=10pt] (M-0-0) at (M-1-1.north west) {};
  \node (M-1-0) at (M-0-0 |- M-1-1) [xshift=-5.5ex] {$\bd^*(y_0)$};
  \node (M-2-0) at (M-0-0 |- M-2-2) [xshift=-5.5ex,yshift=0.5ex] {$d^*_y$};
  \node (M-3-0) at (M-0-0 |- M-3-3) [xshift=-5.5ex] {$d^*_x$};
  \node (M-4-0) at (M-0-0 |- M-4-4) [xshift=-5.5ex] {$\bd^*(y_0') $};
  \node (M-5-0) at (M-0-0 |- M-5-5) [xshift=-5.5ex,yshift=0.5ex] {$d^*_y$};
  \node (M-6-0) at (M-0-0 |- M-6-3) [xshift=-5.5ex] {$d^*_x$};
  \node (M-0-1) at (M-0-0 -| M-1-1) {$d^*_x$};
  \node (M-0-2) at (M-0-0 -| M-2-2) {$d^*_x$};
  \node (M-0-3) at (M-0-0 -| M-3-3) {$d^*_y$};
  \node (M-0-4) at (M-0-0 -| M-4-4) [xshift=-0.75ex] {$\bd^*(y_0)$};
  \node (M-0-5) at (M-0-0 -| M-5-5) [xshift=0.75ex] {$\bd^*(y'_0)$};
  \node (M-0-6) at (M-0-0 -| M-6-6) {$d^*_y$};
\draw[-] (M-3-0.south west) -- (M-3-0.south west -| M-3-6.south east);
   \end{tikzpicture}}}
\end{equation}
where the rectangles around $A_V, B'_V$ indicate that these each take up 2 block rows within the second block column (the only way they fit into this column with their given sizes).

\subsection{Block structure and group action}\label{sect:blockstructure}
We now introduce a block structure on elements of $S(\bd^*)$ 
refining the block structure seen in \eqref{eq:Sd}, 
and a parabolic subgroup $P(\bd^*) \subset G$ acting on $K\backslash G$ by right multiplication. 
We will see that the embedding $\eta$ above is equivariant with respect to an embedding of $GL(\bd^*) \into P(\bd^*)$.

First we consider sequences of vertices of $Q^*(n)$ used to label the block rows and columns of matrices in $S(\bd^*)$.  
We consider an element of $S(\bd^*)$ as an $a \times (a+b)$ matrix $M$ stacked on top of a $b \times (a+b)$ matrix $N$, where $a, b$ are as above.
Afterwards, we will describe the sizes of the blocks, which depend also on $\bd^*$.
To start with, consider the sequences
\[
x=(x_0, x_1, \cdots, x_n), \qquad y=(y_1, \cdots, y_n).
\]
Using $\vartriangleleft$ to denote concatenation of sequences, we label the block rows of $M$ from top to bottom with $(y_0) \vartriangleleft y \vartriangleleft x$, and label the block rows of $N$ from top to bottom with
$(y'_0) \vartriangleleft y \vartriangleleft x$.

The block column labels are slightly more complicated.  Let 
\[
x^s=(x^s_0,  x^s_1, \cdots , x^s_n), \qquad y^s=(y^s_1, \cdots, y^s_n)
\]
be ``second'' copies of the sets $x, y$, and write $\tilde{}$ over a sequence to denote its reversal.  We also define a formal symbol $y_0 + y'_0$.
 Label block columns of both $M$ and $N$ by the sequence
 \begin{equation}\label{eq:Cn}
\cC(n)=\tilde{x} \vartriangleleft x^s \vartriangleleft \tilde{y} \vartriangleleft (y_0 + y'_0) \vartriangleleft  y^s.
 \end{equation}
See Figure \ref{fig:exampleSlice} for an example of this labeling when $n=2$.

When referring to column indices in the remainder of the paper, we consider $\cC(n)$ as a linearly ordered set in the natural way.  For example, we have comparisons
\begin{equation}
x_3 < x_1 < x_2^s < x_3^s < y_3 < y_1 < y_1^s < y_2^s.
\end{equation}

The sizes of the blocks in $S(\bd^*)$ are determined by the labels in the following intuitive way.  
A block row or column labeled by a single vertex $v\in Q^*(n)$ (or $v^s$) has size $\bd^*(v)$, and the block column labeled by $y_0+y_0'$ has size $\bd^*(y_0)+\bd^*(y_0')$.  In examples and figures below, we continue illustrating the single $y_0+y_0'$ block column with two columns, even though they are grouped together.

Finally, we let $P:=P(\bd^*) \subseteq G$ denote the parabolic subgroup of block upper triangular matrices where sizes of the column blocks are the same as those for $S(\bd^*)$.  Thus, right multiplication of a matrix in $P(\bd^*)$ by a matrix representing a point of $S(\bd^*)$ is compatible with the block structure (although $S(\bd^*)$ is not $P(\bd^*)$-stable).

Recall that we denote a typical element of $GL(\bd^*)$ by a list $(g_z)_{z \in Q^*(n)_0}$.  We write 
\[
g_x = {\rm diag}(g_{x_0}, \dotsc, g_{x_n}), \qquad g_y = {\rm diag}(g_{y_1}, \dotsc, g_{y_n})
\]
for the block diagonal embeddings of these sequences of elements in $GL_{d^*_x}$ and $GL_{d^*_y}$, respectively. 
Let $\theta\colon GL(\bd^*) \into P(\bd^*)$ be the map which sends $(g_z)_{z \in Q^*(n)_0}$ to the following matrix.
\begin{equation}\label{eq:blockdiag}
\theta((g_z)):= \begin{blockarray}{cccccccccccc}
          & d^*_x & d^*_x & d^*_y & \bd^*(y_0) & \bd^*(y_0') & d^*_y \\
      \begin{block}{c[ccccccccccc]}
d^*_x & Jg_xJ & \cdot & \cdot & \cdot & \cdot & \cdot \\
d^*_x & \cdot & g_x & \cdot & \cdot & \cdot & \cdot \\
d^*_y & \cdot & \cdot & Jg_y J & \cdot & \cdot & \cdot \\
\bd^*(y_0) & \cdot & \cdot & \cdot & Jg_{y_0} J & \cdot & \cdot \\
\bd^*(y_0') & \cdot & \cdot & \cdot & \cdot & g_{y'_0} & \cdot \\
d^*_y & \cdot & \cdot & \cdot & \cdot & \cdot & g_y  \\
      \end{block}
    \end{blockarray}
\end{equation}
The following lemma can be directly checked by matrix multiplication.
\begin{lemma}\label{lem:closedEmbeddingGroups}
The map $\theta: GL(\bd^*) \into P(\bd^*)$ is a closed embedding of algebraic groups, and $\eta$ is equivariant with respect to this embedding.
\end{lemma}

Let $T(\bd^*) \leq GL(\bd^*)$ be the maximal torus of diagonal matrices in each factor. Then $\theta(T(\bd^*))=T_S$ from Section \ref{sect:linearSliceDefinition}.

\section{Describing intersections of $P$-orbits in $K\backslash G$ with the slice}\label{sect:Smirnov}

Fix a quiver $Q^*(n)$ as in \eqref{eq:Dnquiver} and a dimension vector $\bd^*$ for $Q^*(n)$. Let $G, K, S(\bd^*), P(\bd^*),$ and $\eta: \rep_{Q^*(n)}(\bd^*)\rightarrow S(\bd^*)$ be as in Section \ref{sect:embedding}. Let $O_x\subseteq K\backslash G$ denote the $P(\bd^*)$-orbit through $x\in K\backslash G$, and let $\overline{O}_x$ denote its closure in $K\backslash G$.

In this section, we describe the intersections $O_x\cap S(\bd^*)$ and $\overline{O}_x\cap S(\bd^*)$ in terms of certain $\mathbb{Z}$-valued functions. Our main result is that if $x = \eta(V)$ for some representation $V\in \rep_{Q^*(n)}(\bd^*)$, then these intersections are characterized in terms of the $\mathbb{Z}$-valued functions $\mathcal{D}_n$ from \eqref{eq:distCollection}. See Proposition \ref{prop:eta1Orbit} for a precise statement. 

\subsection{$\mathbb{Z}$-valued functions determining orbits and section overview}\label{sect:SmirnovDef}
It is well-known that there are finitely many $B_+$-orbits (and therefore also $P(\bd^*)$-orbits) in $K\backslash G$, and that furthermore the $P(\bd^*)$-orbit through a given point in $K\backslash G$ is characterized by a collection of matrix ranks (see for example \cite{Smirnov08}).  

It is also true, though significantly more difficult\footnote{Since $B$-orbit closures in $K \backslash G$ are not linearly ordered, it is not even a priori clear that these rank inequalities define irreducible subvarieties. } to prove that orbit closure containment can be characterized by rank inequalities. 
This characterization is Theorem \ref{thm:Porbitranks} below. It can be obtained by translating work of Wyser \cite[Cor.~1.3]{Wyser16} to our setup, or applying Bongartz's result \cite[Thm.~3.3]{Bongartz96} in conjunction with the ideas of Theorem \ref{thm:GKinD}.

In what follows, we use the matrix definitions from Section \ref{sect:prelimMatrixDefs}. 
A pair of matrices $(M,N)$, with $M \in \Mat(a,a+b), N \in \Mat(b,a+b)$, will be used to denote a point of $\Mat(a+b,a+b)$ by stacking $M$ on top of $N$.  When the resulting matrix is invertible, we also use it to represent a point of $K\backslash G$.
Given $v\in \cC(n)$ (for $\cC(n)$ as in \eqref{eq:Cn}),  define the $\mathbb{Z}$-valued functions $\mathscr{U}_v$ and $\mathscr{L}_v$ on $K\backslash G$ by $\mathscr{U}_v(M,N) := \rank M_{[x_n,v]}$ and $\mathscr{L}_v(M,N) := \rank N_{[x_n,v]}$. Given $v < w$ in $\cC(n)$, define the $\mathbb{Z}$-valued function $\mathscr{B}_{v,w}$ on $K\backslash G$ by $\mathscr{B}_{v,w}(M,N) := \rank (M,N)_{[[v,w]]}$. 
Note that $\mathscr{U}$, $\mathscr{L}$, and $\mathscr{B}$ stand for ``upper'', ``lower'', and ``both'' respectively.

\begin{theorem}\label{thm:Porbitranks}
Let $(M,N), (\tilde{M},\tilde{N})\in K\backslash G$.   
The containment of $P(\bd^*)$-orbit closures $\overline{O}_{(M,N)}\subseteq \overline{O}_{(\tilde{M},\tilde{N})}$ holds if and only if 
\begin{enumerate}[(i)]
\item $\mathscr{U}_v(M,N) \leq \mathscr{U}_v(\tilde{M},\tilde{N})$ and $\mathscr{L}_v(M,N) \leq \mathscr{L}_v(\tilde{M},\tilde{N})$ for all $v \in \cC(n)$; and
\item $\mathscr{B}_{v,w}(M,N) \leq  \mathscr{B}_{v,w}(\tilde{M},\tilde{N})$ for all $v < w$ in $\cC(n)$.
\end{enumerate}
The equality of $P(\bd^*)$-orbits $O_{(M,N)} = O_{(\tilde{M},\tilde{N})}$ holds if and only if (i) and (ii) hold upon replacing each inequality of function values by an equality of function values. 
\end{theorem}

In the remainder of this section, we systematically consider the restrictions of the functions $\mathscr{U}_v$, $\mathscr{L}_v$, and $\mathscr{B}_{v,w}$ to $S(\bd^*)$.
 We partition these functions into three classes:
\begin{description}
\item[Constant] those which are constant on $S(\bd^*)$ (Section \ref{sect:constant});
\item[Image] a subset whose values we fix to cut out the image of $\eta$ (Section \ref{sect:specializeTheSlice});
\item[Quiver] a subset which corresponds to quiver rank conditions (Section \ref{sect:remainingRanks}).
\end{description} 

Figures \ref{fig:ranktable1} and \ref{fig:ranktable2} illustrate this partition for the case $n=2$:
the constant type functions are marked with $C$, the image type functions are marked with $Im$, and the remaining quiver type functions are marked with a function in the set \eqref{eq:distCollection}.

The proposition below is the main result of this section and a key ingredient in the proof of Theorem \ref{thm:mainTheorem}.
Its proof is found at the end of Section \ref{sect:remainingRanks}.

\begin{proposition}\label{prop:eta1Orbit}
Let $V, W\in \rep_{Q^*(n)}(\bd^*)$ and let $\mathcal{D}_n$ be as in \eqref{eq:distCollection}. Then $\overline{O}_{\eta(V)}\cap S(\bd^*)$ is a reduced and irreducible subvariety of $\im (\eta)$. Furthermore, we have
\begin{enumerate}[(i)]
\item $\overline{O}_{\eta(V)}\cap S(\bd^*)\subseteq \overline{O}_{\eta(W)}\cap S(\bd^*)$ if and only if $f(V) \leq f(W)$, for all $f\in \mathcal{D}_n$, and
\item $O_{\eta(V)}\cap S(\bd^*) = O_{\eta(W)}\cap S(\bd^*)$ if and only if $f(V) = f(W)$, for all $f\in \mathcal{D}_n$.
\end{enumerate}
\end{proposition}

The next three sections establish preliminary results, whose proofs are essentially all direct checks that ranks of certain matrices are equal.
We warn the reader that these sections are rather technical, resulting from a long trial-and-error process rather than an easily summarizable intuition.
Although we provide a proof for arbitrary $n$, simply checking the results for the $n=2$ case using Figures \ref{fig:ranktable1} and \ref{fig:ranktable2} is likely to provide the reader with a sufficient understanding of the ideas.  In particular, already in this case one can see that the partition of the functions $\mathscr{U}_v, \mathscr{L}_v$, and $\mathscr{B}_{v,w}$ into the three types (constant, image, or quiver) is not straightforward, and keeping track of the partition for the functions $\mathscr{B}_{v,w}$ is particularly inconvenient.  
We warn the reader that the order in which the various functions are considered is essential to understanding the closed subvarieties they cut out.   
In Section \ref{sect:ind}, we show how to use induction to reduce the number of direct checks necessary.

\begin{figure}
\[
\begin{array}{c|c|c|c|c|c|c|c|c|c|c|c|}
v  & x_2&x_1 & x_0 & x^s_0 & x^s_1 & x^s_2 & y_2 & y_1 & y_0 + y'_0 & y^s_1 & y^s_2\\
\hline
\mathscr{U}_v &C& C & C & |\zb_0,\zb_0| & |\zb_0,\zb_1| & |\zb_0,\zb_2| & |\zb_0,\za_2| & |\zb_0,\za_1| & C & C & C \\
\hline
\mathscr{L}_v &C& C & C & Im & |\za'_1,\zb_1| & |\za'_1,\zb_2| & |\za'_1,\za_2| & |\za'_1,\za'_1| & C & C & C \\
\hline
\end{array}
\]
\caption{$\mathscr{U}_v$ and $\mathscr{L}_v$ for $n=2$.}\label{fig:ranktable1}
\end{figure}

\begin{figure}
\[
\begin{array}{c|c|c|c|c|c|c|c|c|c|c|}
v\quad  \backslash w  & x_1 & x_0 & x^s_0 & x^s_1 & x^s_2 & y_2 & y_1 & y_0 + y'_0 & y^s_1 & y^s_2\\
\hline
x_2 & C & C & C & C & |\za_2,\zb_2| & |\za_2,\za_2| & \red{Im} & \red{Im} & \red{Im} & C\\
\hline
x_1 & & C & C & || \alpha_1,\zb_1||^0 & || \za_1,\zb_2||^0& || \za_1,\za_2||^0 & || \za_1,\za_1||^0 & Im & \red{Im} & C\\
\hline
x_0  & & & |\zb_0, \zb_0| & || \za_1,\zb_1|| & || \za_1,\zb_2|| & || \za_1,\za_2|| & || \za_1,\za_1|| & Im & \red{Im} & C\\
\hline
x^s_0 & & & & ||\za_1, \zb_1|| & ||\za_1, \zb_2|| & ||\za_1, \za_2|| & ||\za_1, \za_1|| & Im & \red{Im} & C\\
\hline
x^s_1 & & & & & ||\za_2, \zb_2|| & ||\za_2, \za_2|| & ||\zb_1, \zb_1|| & |\zb_1, \zb_1| & \red{Im} & C\\
\hline
x^s_2 & & & & & & ||\zb_2, \zb_2|| & ||\zb_1, \zb_2|| & |\zb_1, \zb_2| & |\zb_2,\zb_2| & C\\
\hline
y_2 & & & & & & & ||\zb_1, \za_2|| & |\zb_1, \za_2| & C & C\\
\hline
y_1 & & & & & & & & C & C & C\\
\hline
y_0+y'_0 & & & & & & & & & C & C\\
\hline
y^s_1 & & & & & & & & & & C\\
\hline
\end{array}
\]
\caption{$\mathscr{B}_{v,w}$ for $n=2$. Red text indicates functions which become constant when restricted to $R(\bd)$ (see Section \ref{sect:ind}, Lemma \ref{lem:toRn}).}\label{fig:ranktable2}.
\end{figure}

\subsection{Constant type functions}\label{sect:constant}
The following proposition identifies which $\mathscr{U}_v$, $\mathscr{L}_v$, and $\mathscr{B}_{v,w}$  are constant when restricted to $S(\bd^*)$ (marked with $C$ in Figures \ref{fig:ranktable1} and \ref{fig:ranktable2}).

\begin{proposition}\label{prop:charRanks}
The following functions on $K\backslash G$ are constant when restricted to $S(\bd)$:
\begin{enumerate}[(i)]
\item $\mathscr{U}_v$ and $\mathscr{L}_v$ for $v\leq x_0$ or $v\geq y_0+y_0'$;
\item $\mathscr{B}_{v,y_n^s}$, for $v<y_n^s$;
\item  $\mathscr{B}_{x_i,w}$, for $x_i<w<x_i^s$, $1\leq i\leq n$;
\item $\mathscr{B}_{y_i, w}$, for $w\geq y_{i-1}^s$, $2\leq i\leq n$; 
\item $\mathscr{B}_{v,w}$, $y_1\leq v<w $. 
\end{enumerate}
\end{proposition}

\begin{proof}
A straightforward direct check shows that each above-listed function attains its maximum possible value when evaluated at any $(M,N)\in S(\bd^*)$. 
\end{proof}

We refer to all remaining $\mathscr{U}_v$, $\mathscr{L}_v$, and $\mathscr{B}_{v,w}$ as \emph{non-constant type functions}, despite having not yet proven that they are indeed non-constant on $S(\bd^*)$.

\subsection{An inductive approach to non-constant type functions}\label{sect:ind}
Fix $n\geq 2$ and let $\bd'$ denote the restriction of the dimension vector $\bd^*$ for $Q^*(n)$ to the subquiver $Q^*(n-1)$. 
In this section, we define a linear subvariety $R(\bd^*)\subseteq S(\bd^*)$ and show that the restriction of each function $\mathscr{B}_{v,w}$ of non-constant type from $S(\bd^*)$ to $R(\bd^*)$ is either a constant function (Lemma \ref{lem:toRn}, and marked in red in Figure \ref{fig:ranktable2}) or is equal to a corresponding non-constant type function on $S(\bd')$, up to a constant which depends on $\bd^*$, $v$, and $w$ (Proposition \ref{prop:inductionSlice}). 

As a visual aid, Figure \ref{fig:exampleSlice} shows the slice along with the row and column labels for $n=2$.
Figure \ref{fig:exampleSlice2} shows the passage to $R(\bd^*)$ (Lemma \ref{lem:toRn}) then $\im(\eta)$ (Lemma \ref{lem:specSlice}) by successively fixing the values of a subset of the non-constant type $\mathscr{U}_v, \mathscr{L}_v, \mathscr{B}_{v,w}$.

\begin{figure}
\[
S(\bd^*): \qquad \begin{blockarray}{cccccccccccccccc}
         &  x_2 & x_1 & x_0 & x_0^s & x_1^s & x_2^s & y_2 & y_1 & y_0 + y'_0 & y_1^s & y_2^s  \\
      \begin{block}{c[ccccccccccccccc]}
y_0 & \cdot & \cdot & \cdot & * & * & * & \cdot & \cdot &  J \qquad \cdot   & \cdot & \cdot \\
y_1 & \cdot & \cdot & \cdot & * & * & * & \cdot & J &\cdot \qquad \cdot & \cdot & \cdot \\
y_2 & \cdot & \cdot & \cdot & * & * & * & J & \cdot & \cdot \qquad \cdot & \cdot & \cdot \\
x_0 & \cdot & \cdot & J & \cdot & \cdot & \cdot & \cdot & \cdot & \cdot \qquad \cdot & \cdot & \cdot \\
x_1 & \cdot & J & \cdot & \cdot & \cdot & \cdot & \cdot & \cdot & \cdot \qquad \cdot & \cdot & \cdot \\
x_2 & J & \cdot & \cdot & \cdot & \cdot & \cdot & \cdot & \cdot & \cdot \qquad \cdot & \cdot & \cdot \\
\cline{1-15}
y'_0 & \cdot & \cdot & \cdot & \star & \star & \star & \cdot & \cdot & \cdot \qquad 1 & \cdot & \cdot \\
y_1 & \cdot & \cdot & \cdot & \star & \star & \star & \cdot & J & \cdot \qquad \cdot & 1 & \cdot \\
y_2 & \cdot & \cdot & \cdot & \star & \star & \star & J & \cdot & \cdot \qquad \cdot & \cdot & 1\\
x_0 & \cdot & \cdot & J & 1 & \cdot & \cdot & \cdot & \cdot & \cdot \qquad \cdot & \cdot & \cdot \\
x_1 & \cdot & J & \cdot & \cdot & 1 & \cdot & \cdot & \cdot & \cdot \qquad \cdot & \cdot & \cdot \\
x_2 & J & \cdot & \cdot & \cdot & \cdot & 1 & \cdot & \cdot & \cdot \qquad \cdot & \cdot & \cdot \\
      \end{block}
    \end{blockarray}    
\]
\caption{$S(\bd^*)$ for $n=2$}\label{fig:exampleSlice}
\end{figure}
\begin{figure}
{\small
\[
\begin{bmatrix}
\cdot & \cdot & \cdot & * & * & \red{0} & \cdot & \cdot & J \qquad \cdot   & \cdot & \cdot \\
\cdot & \cdot & \cdot & * & * & \red{M_{y_1,x_2^s}} & \cdot & J  & \cdot \qquad \cdot & \cdot & \cdot \\
\cdot & \cdot & \cdot & \red{0} & \red{0} & \red{M_{y_2,x_2^s}} & J  & \cdot & \cdot \qquad \cdot & \cdot & \cdot \\
\cdot & \cdot & J & \cdot & \cdot & \cdot & \cdot & \cdot & \cdot \qquad \cdot & \cdot & \cdot \\
\cdot & J & \cdot & \cdot & \cdot & \cdot & \cdot & \cdot & \cdot \qquad \cdot & \cdot & \cdot \\
J & \cdot & \cdot & \cdot & \cdot & \cdot & \cdot & \cdot & \cdot \qquad \cdot & \cdot & \cdot \\
\hline
\cdot & \cdot & \cdot & \star & \star & \red{0} & \cdot & \cdot & \cdot \qquad 1 & \cdot & \cdot \\
\cdot & \cdot & \cdot & \star & \star & \red{M_{y_1,x_2^s}} & \cdot & J  & \cdot \qquad \cdot & 1 & \cdot \\
\cdot & \cdot & \cdot & \red{0} & \red{0} & \red{M_{y_2,x_2^s}} & J & \cdot & \cdot \qquad \cdot & \cdot & 1\\
\cdot & \cdot & J & 1 & \cdot & \cdot & \cdot & \cdot & \cdot \qquad \cdot & \cdot  & \cdot \\
\cdot & J & \cdot & \cdot & 1 & \cdot & \cdot & \cdot & \cdot \qquad \cdot & \cdot  & \cdot \\
J & \cdot & \cdot & \cdot & \cdot & 1 & \cdot & \cdot & \cdot \qquad \cdot & \cdot  & \cdot \\
\end{bmatrix}
\quad
\begin{bmatrix}
\cdot & \cdot & \cdot & \red{V_{\beta_0}} & \red{V_{\alpha_1}} & \red{0} & \cdot & \cdot & J \qquad \cdot & \cdot & \cdot \\
\cdot & \cdot & \cdot & \red{0} & \red{V_{\beta_1}} & \red{V_{\alpha_2}} & \cdot & J & \cdot \qquad \cdot & \cdot & \cdot\\
\cdot & \cdot & \cdot & \red{0} & \red{0} & \red{V_{\beta_2}} & J & \cdot  & \cdot \qquad \cdot & \cdot & \cdot \\
\cdot & \cdot & J & \cdot & \cdot & \cdot & \cdot & \cdot & \cdot \qquad \cdot & \cdot & \cdot \\
\cdot & J & \cdot & \cdot & \cdot & \cdot & \cdot & \cdot & \cdot \qquad \cdot & \cdot & \cdot \\
J & \cdot & \cdot & \cdot & \cdot & \cdot & \cdot & \cdot & \cdot \qquad \cdot & \cdot & \cdot \\
\hline
\cdot & \cdot & \cdot & \red{0} & \red{V_{\alpha'_1}} & \red{0} & \cdot & \cdot & \cdot \qquad 1 & \cdot & \cdot \\
\cdot & \cdot & \cdot & \red{0} & \red{V_{\beta_1}} & \red{V_{\alpha_2}} & \cdot & J & \cdot \qquad \cdot & 1 & \cdot \\
\cdot & \cdot & \cdot & \red{0} & \red{0} & \red{V_{\beta_2}} & J & \cdot & \cdot \qquad \cdot & \cdot & 1\\
\cdot & \cdot & J & 1 & \cdot & \cdot & \cdot & \cdot & \cdot \qquad \cdot & \cdot & \cdot \\
\cdot & J & \cdot & \cdot & 1 & \cdot & \cdot & \cdot & \cdot \qquad \cdot & \cdot & \cdot \\
J & \cdot & \cdot & \cdot & \cdot & 1 & \cdot & \cdot & \cdot \qquad \cdot & \cdot & \cdot \\
\end{bmatrix}
\]
\[
\hfill R(\bd^*) \hspace{7cm} \im(\eta) \hfill
\]
}
\caption{The shapes of $R(\bd^*)$, and $\im \eta$ for $n=2$}\label{fig:exampleSlice2}
\end{figure}

\begin{definition}\label{def:Rn}
Let $R(\bd^*)\subseteq S(\bd^*)$ consist of those $(M,N)\in S(\bd^*)$ such that: 
\begin{enumerate}
\item[(R1)] $M_{v,x^s_{n}} = N_{v,x^s_{n}}$, $y_{1}\leq v\leq y_n$,
\item[(R2)] $M_{v,x_n^s} = 0$ for all $y_0\leq v\leq y_{n-2}$ and $N_{v,x_n^s} = 0$ for all $y_0'\leq v\leq y_{n-2}$,
\item[(R3)] $M_{y_n,w} = N_{y_n,w}$, for all $x_0^s\leq w\leq x^s_{n}$,
\item[(R4)] $M_{y_n,w} = N_{y_n,w}=0$, for all $x_0^s\leq w\leq x^s_{n-1}$. \hfill \qedhere
\end{enumerate}
\end{definition}
There is intentional redundancy in Definition \ref{def:Rn} to make it easy to check that $R(\bd^*)$ can be obtained from $S(\bd^*)$ by fixing the values of some of the image type functions.

Define $b_{v,w}:= \text{min}\{\mathscr{B}_{v,w}(M,N)\mid (M,N)\in S(\bd^*)\}$.

\begin{lemma}\label{lem:toRn}
Let $(M,N)\in S(\bd^*)$. Then $(M,N)\in R(\bd^*)$ if and only if $\mathscr{B}_{v,w}(M,N) = b_{v,w}$ for the following ordered pairs $(v,w)$:
\begin{enumerate}[(i)]
\item $(x_n,w)$, $y_0+y_0'\leq w\leq y_{n-1}^s$; 
\item $(x_n,w)$, $y_{n-1}\leq w\leq y_1$; 
\item $(v,y_{n-1}^s)$, $x_n\leq v\leq x_0$; and
\item $(v,y_{n-1}^s)$, $x_0^s\leq v\leq x^s_{n-1}$.
\end{enumerate}
\end{lemma}

\begin{proof}
A direct check shows (i) and (iii) of Lemma \ref{lem:toRn} hold 
if and only if $(M,N)$ satisfies (R1) and (R3) of Definition \ref{def:Rn} respectively.
Then, assuming that $(M,N)$ satisfies (R1) (respectively (R3)) of Definition \ref{def:Rn}, a direct check shows that (ii) (respectively (iv)) of Lemma \ref{lem:toRn} holds 
if and only if $(M,N)$ satisfies (R2) (respectively (R4)) of Definition \ref{def:Rn}.
\end{proof}

Given $(M,N)\in R(\bd^*)$, we let $(M', N')\in S({\bd'})$ be the pair of matrices obtained from $(M,N)$ by deleting all block rows and block columns indexed by $x_n, x_n^s, y_n, y_n^s$. 

\begin{proposition}\label{prop:inductionSlice}
Fix a dimension vector $\bd^*$ and pair $v< w$ in $\cC(n-1)$.  Then there exists a constant $\xi(\bd^*, v, w)$ such that for all $(M,N)\in R(\bd^*)$ we have
\begin{equation}\label{eq:inductionSlice}
\mathscr{B}_{v,w}(M,N) = \mathscr{B}_{v,w}(M',N') + \xi(\bd^*, v, w).
\end{equation}
where $\mathscr{B}_{v,w}$ denotes a function on $S(\bd^*)$ on the left of \eqref{eq:inductionSlice}, and a function on $S(\bd')$ on the right of \eqref{eq:inductionSlice}.
\end{proposition}

\begin{proof}
Let $(M,N)\in R(\bd^*)$ and label blocks $M_{y_{n-1},x_n^s}$, $N_{y_{n-1},x_n^s}$ (which are equal to one another) by $V_{\alpha_n}$, and label blocks $M_{y_n,x_n^s}, N_{y_n,x_n^s}$ by $V_{\beta_n}$. 

If $w\leq x_{n-1}^s$ then the result is easy to check. 
The remaining cases are (i) $w\geq y_{n-1}$, $v\leq x_{n-1}^s$ and (ii) $w\geq y_{n-2}$, $v\geq y_{n-1}$. 
In each of case, use the $J$s originating from the $y_n$ and $y_{n-1}$ block columns of $(M,N)$ to clear all appearances  of $V_{\alpha_n}$ and $V_{\beta_n}$  in $(M,N)_{[[v,w]]}$ using column operations. Then transform the resulting matrix, by swapping the orders of some rows and columns, to the direct sum of $(M',N')_{[[v,w]]}$ with identity and zero matrices.

Since the ranks of two matrices which are row or column equivalent are the same, and $\mathscr{B}_{v,w}(M,N)$ is the rank of the matrix $(M,N)_{[[v,w]]}$, the result follows.
\end{proof}

\subsection{Image type functions}\label{sect:specializeTheSlice}
It is immediate from the zig-zag shapes of $A_V$ and $B_V$ in \eqref{eq:M}, and the definition of $\eta$, that $\im(\eta)$ consists of those $(M,N)\in S(\bd^*)$ satisfying:
\begin{enumerate}
\item[(I1)] (part of $M$ is equal to part of $N$) $M_{v,w} = N_{v,w}$, for $y_1\leq v\leq y_n$, $x_0^s\leq w\leq x_n^s$,
\item[(I2)] (entries below a zig-zag are zero) $M_{y_i, x_j^s}=0$ and $N_{y_i, x_j^s}=0$, for $1\leq j<i\leq n$,
\item[(I3)] (entries above a zig-zag are zero) $M_{y_i, x_j^s}=0 $ and $N_{y_i, x_j^s}=0$, for $3\leq i+2\leq j\leq n$, and $M_{y_0, x_i^s}=0$ and $N_{y_0', x_i^s}=0$, for $i\geq 2$,
\item[(I4)] (part of the $x_0^s$ column of $N$ is zero) $N_{y_0',x_0^s} = 0$ and $N_{y_i,x_0^s} = 0$, for $i\geq 1$.
\end{enumerate}
An example of $\im \eta$ for the case $n=2$ is seen on the right of Figure \ref{fig:exampleSlice2}.

The following lemma characterizes those functions $\mathscr{U}_{v}$, $\mathscr{L}_v$, and $\mathscr{B}_{v,w}$ whose values we fix to cut out $\im(\eta)$. These are the image type functions. 
As above, let $b_{v,w}$ be the minimum value of $\mathscr{B}_{v,w}$ attained on $S(\bd^*)$. 
Note also that $d^*_x$ is the minimum value of $\mathscr{L}_{x_0^s}$ attained on $S(\bd^*)$.
\begin{lemma}\label{lem:specSlice}
Let $(M,N)\in S(\bd^*)$.  
Then $(M,N)\in \im(\eta)$ if and only if $\mathscr{L}_{x_0^s}(M,N) = d^*_x$ and $\mathscr{B}_{v,w}(M,N) = b_{v,w}$ for the following ordered pairs $(v,w)$:
\begin{enumerate}[(i)]
\item $(v, w)$, $x_n\leq v\leq x_0$, $y_0+y_0'\leq w\leq y^s_{n-1}$;
\item $(x_0^s, w)$, $y_0+y_0'\leq w\leq y_{n-1}^s$, and $(x_i^s, y_j^s)$, $1\leq i \leq j\leq n-1$; and
\item $(x_i,y_j)$, $1\leq j < i\leq n$.
\end{enumerate}
\end{lemma}

\begin{proof}
We prove the lemma by induction on $n$.
Let $n = 1$ and fix a dimension vector $\bd^*$ for $Q^*(1)$. Each $(M,N)\in S(\bd^*)$ has the form:
\begin{equation}\label{eq:sliceSpecialization}
  \begin{blockarray}{cccccccc}
         & x_1 & x_0 & x_0^s & x_1^s & y_1 & y_0 + y'_0 & y_1^s \\
      \begin{block}{c[ccccccc]}
y_0 & \cdot & \cdot & *_1& *_3 & \cdot & J\qquad \cdot  &\cdot\\
y_1 & \cdot & \cdot & *_2 & *_4 & J & \cdot \qquad \cdot & \cdot\\
x_0 & \cdot & J & \cdot & \cdot & \cdot & \cdot \qquad \cdot & \cdot \\
x_1 & J & \cdot & \cdot & \cdot & \cdot & \cdot \qquad \cdot & \cdot\\
\cline{1-8}
y'_0  & \cdot & \cdot &\star_1 & \star_3 & \cdot & \cdot \qquad 1 & \cdot\\
y_1 & \cdot & \cdot & \star_2 & \star_4 & J & \cdot \qquad \cdot &1 \\
x_0 & \cdot & J & 1 & \cdot & \cdot & \cdot \qquad \cdot & \cdot \\
x_1 & J & \cdot & \cdot & 1 & \cdot & \cdot \qquad \cdot & \cdot \\
      \end{block}
    \end{blockarray}
    \end{equation}
There are two functions $\mathscr{B}_{v,w}$ arising from (i) in the lemma statement, namely $\mathscr{B}_{x_1,y_0+y_0'}$ and $\mathscr{B}_{x_0,y_0+y_0'}$. A direct check shows that $\mathscr{B}_{x_1,y_0+y_0'}(M,N) = b_{x_1,y_0+y_0'}$ 
 if and only if $*_4 = \star_4$, and $\mathscr{B}_{x_0,y_0+y_0'}(M,N) = b_{x_0,y_0+y_0'}$ if and only if both $*_4 = \star_4$ and $*_2 = \star_2$ (i.e. condition (I1) holds). There is one function arising from (ii) of the lemma statement, namely $\mathscr{B}_{x_0^s,y_0+y_0'}$. Assuming that (I1) holds, $\mathscr{B}_{x_0^s,y_0+y_0'}(M,N) = b_{x_0^s,y_0+y_0'}$  if and only if $*_2 = \star_2 = 0$ (i.e. (I2) holds). Condition (iii) of the present lemma and (I3) are vacuous. 
 Finally, $\mathscr{L}_{x_0^s}(M,N) = d_x^*$ if and only if $\star_1=0$ and $\star_2 = 0$ (i.e. (I4) holds). 
This completes the proof of the base case. 

Now assume that $n\geq 2$, fix a dimension vector $\bd^*$ for $Q^*(n)$, and let $(M,N)\in S(\bd^*)$. 
Let $\bd'$ be the dimension vector obtained by restricting $\bd^*$ to the subquiver $Q^*({n-1})$, and $\eta'\colon \rep_{Q^*(n-1)}(\bd')\rightarrow S(\bd')$ the closed embedding of \eqref{eq:defEta} for this smaller quiver and dimension vector.  
Let $(M',N')\in S(\bd')$ be obtained from $(M,N)$ by removing all rows and columns indexed by $x_n, x_n^s, y_n, y_n^s$. 
We can directly observe that $(M,N) \in \im(\eta)$ if and only if (I4) holds, $(M,N) \in R(\bd^*)$, and $(M',N') \in \im(\eta')$. So, we will prove that (I4) holds, $(M,N) \in R(\bd^*)$, and $(M',N') \in \im(\eta')$ if and only if all conditions listed in the present lemma hold.

We begin with the forward direction, and assume that (I4) holds, $(M,N) \in R(\bd^*)$, and $(M',N') \in \im(\eta')$. Since (I4) holds, we have by an easy direct check that $\mathscr{L}_{x_0^s}(M,N) = d_x^*$. Since $(M,N) \in R(\bd^*)$, we have that (i) through (iv) of Lemma \ref{lem:toRn} hold. 
The remaining conditions in the list from the present lemma are: 
\begin{enumerate}[(i)]
\item $\mathscr{B}_{v, w}(M,N) = b_{v,w}$, $x_{n-1}\leq v\leq x_0, y_0+y_0'\leq w\leq y^s_{n-2}$;
\item $\mathscr{B}_{x_0^s, w}(M,N) = b_{x_0^s,w}$, $y_0+y_0'\leq w\leq y_{n-2}^s$, and $\mathscr{B}_{x_i^s, y_j^s}(M,N)\! =\! b_{x_i^s, y_j^s}$, $1\leq i \leq j\leq n-2$;
\item $\mathscr{B}_{x_i,y_j}(M,N) = b_{x_i,y_j}$, $1\leq j < i\leq n-1$.
\end{enumerate}
Since $(M',N')\in \im(\eta')$, the induction hypothesis implies that these conditions hold upon replacing $(M,N)$ in the above list by $(M',N')$. Then, since $(M,N)\in R(\bd^*)$, Proposition \ref{prop:inductionSlice} implies that these conditions also hold for $(M,N)$. The proof of the converse is similar.
\end{proof}

\subsection{Quiver type functions}\label{sect:remainingRanks}
We now consider all remaining functions $\mathscr{U}_v, \mathscr{L}_v, \mathscr{B}_{v,w}$ from Theorem \ref{thm:Porbitranks} (i.e. those which are neither of constant type nor of image type). These are the quiver type functions. 
Let $\mathcal{E}_n$ denote the set of restrictions of quiver type functions to $\im(\eta)$, and recall $\mathcal{D}_n$ from \eqref{eq:distCollection}.
In this section, we will show that there is a surjection $\phi: \mathcal{E}_n \rightarrow \mathcal{D}_n$ such that for each $f\in \mathcal{E}_n$, the difference $f\circ \eta -\phi(f)$ is a constant function on $\rep_{Q^*(n)}(\bd^*)$. We will then prove Proposition \ref{prop:eta1Orbit}, making use of this surjection.

If $f\in \mathcal{E}_n$ and $g\in \mathcal{D}_n$, we use the shorthand $f\sim g$ to mean that $f\circ \eta-g$ is a constant function on $\rep_{Q^*(n)}(\bd^*)$, in which case we say that $f$ and $g$ are \emph{equivalent functions} on $\rep_{Q^*(n)}(\bd^*)$. We also abuse notation and write $\mathscr{U}_v$ (respectively $\mathscr{L}_v$, $\mathscr{B}_{v,w}$) for $\mathscr{U}_v|_{\im(\eta)}$ (respectively $\mathscr{L}_v|_{\im(\eta)}$, $\mathscr{B}_{v,w}|_{\im(\eta)}$) in Lemmas \ref{lem:quiverConditionsOnSlice1p} and \ref{lem:quiverConditionsOnSlice2} below.

The following lemma covers all quiver type functions of the form $\mathscr{U}_v$ and $\mathscr{L}_v$ in $\mathcal{E}_n$, and follows from direct observation (see the right matrix of Figure \ref{fig:exampleSlice2} for the $n=2$ example).

\begin{lemma}\label{lem:quiverConditionsOnSlice1p}
We have the following equivalences of functions:
\begin{enumerate}[(i)]
\item $\mathscr{U}_{x_i^s}\sim |\beta_0,\beta_i|$, $0\leq i\leq n$, and $\mathscr{U}_{y_i}\sim |\beta_0,\alpha_i|$, $1\leq i\leq n$,
\item $\mathscr{L}_{x_i^s}\sim |\alpha_1',\beta_i|$, $1\leq i\leq n$, $\mathscr{L}_{y_1}\sim |\alpha_1',\alpha_1'|$, and $\mathscr{L}_{y_i}\sim |\alpha_1',\alpha_i|$, $2\leq i\leq n$.
\end{enumerate}
\end{lemma}

The remaining elements of $\mathcal{E}_n$ (those of the form $\mathscr{B}_{v,w}$) all appear in the following lemma.

\begin{lemma}\label{lem:quiverConditionsOnSlice2}
We have the following equivalences of functions:
\begin{enumerate}[(i)]
\item $\mathscr{B}_{x_i,x_j^s}\sim |\alpha_i,\beta_j|$, $2\leq i\leq j\leq n$, $\mathscr{B}_{x_1,x_j^s}\sim || \za_1,\beta_j ||^0$, $1\leq j\leq n$, $\mathscr{B}_{x_0,x_0^s}\sim | \zb_0,\beta_0 |$, $\mathscr{B}_{x_0,x_j^s}\sim || \za_1,\beta_j ||$, $1\leq j\leq n$,

\item $\mathscr{B}_{x_i,y_j}\sim |\alpha_i,\alpha_j|$, $2\leq i\leq j\leq n$, $\mathscr{B}_{x_1,y_j}\sim || \za_1,\alpha_j ||^0$, $1\leq j\leq n$, $\mathscr{B}_{x_0,y_j}\sim ||\za_1,\alpha_j ||$, $1\leq j\leq n$,

\item $\mathscr{B}_{x_{i-1}^s,x_j^s}\sim ||\alpha_i,\beta_j ||$, $1\leq i\leq j\leq n$,

\item $\mathscr{B}_{x_{i-1}^s,y_j}\sim ||\alpha_i,\alpha_j ||$, $1\leq i\leq j\leq n$, $\mathscr{B}_{x_j^s,y_i}\sim ||\beta_i,\beta_j ||$, $1\leq i\leq j\leq n$,

\item $\mathscr{B}_{x_j^s,y_0+y_0'}\sim |\beta_1, \beta_j|$, $1\leq j\leq n$ and $\mathscr{B}_{x_j^s,y^s_{i-1}}\sim |\beta_i, \beta_j|$, $2\leq i\leq j\leq n$,
\item $\mathscr{B}_{y_j,y_i}\sim ||\beta_i,\alpha_j ||$, $1\leq i<j\leq n$,
\item $\mathscr{B}_{y_j,y_0+y_0'}\sim |\beta_1, \alpha_j|$, $2\leq j\leq n$ and $\mathscr{B}_{y_j,y^s_{i-1}}\sim |\beta_i, \alpha_j|$, $2\leq i<j\leq n$, 
\end{enumerate}
\end{lemma}

\begin{proof}
We induct on $n$, with $n=1$ being easy to check. 
Let $n\geq 2$, $\bd^*$ a dimension vector for $Q^*(n)$, and $(M,N)\in \im(\eta)$. 
We consider the $\mathscr{B}_{v,w}$ appearing in the lemma statement (i.e. the elements of $\mathcal{E}_n$ of the form $\mathscr{B}_{v,w}$) in two separate cases, the first being where $v,w\neq x_n, x_n^s, y_n, y_n^s$ and the second being where at least one of $v,w$ is some $x_n, x_n^s, y_n, y_n^s$. 

To treat the first case, let $\bd'$ be the dimension vector obtained by restricting $\bd^*$ to $Q^*(n-1)$, let $(M',N')$ the matrix obtained by removing all rows and columns from $(M,N)$ indexed by $x_n,x_n^s, y_n, y_n^s$, and $\eta'\colon \rep_{Q^*(n-1)}(\bd')\rightarrow S(\bd')$ the closed embedding of \eqref{eq:defEta} for this smaller quiver and dimension vector.   
Observe that $(M',N')\in \im(\eta')$, so we may apply the induction hypothesis to see that, on $\rep_{Q^*(n-1)}(\bd')$, the claimed equivalences of functions hold. 
That they also hold on $\rep_{Q^*(n)}(\bd)$ follows from Proposition \ref{prop:inductionSlice}. 

To complete the proof, we consider all possible $\mathscr{B}_{v,w}$, restricted to $\im(\eta)$, when at least one of $v,w$ is $x_n, x_n^s, y_n, y_n^s$. 
We summarize the result for each case.  They can be systematically checked by drawing out the relevant matrices and clearing block rows and column using identity blocks. We first consider the possibilities for $v$, then those for $w$.
The functions of the form $\mathscr{B}_{x_n,w}$ are all constant except
\[
\mathscr{B}_{x_n,x_n^s} \sim |\za_n,\zb_n| \qquad \text{and} \qquad \mathscr{B}_{x_n,y_n}\sim |\za_n,\za_n|.
\]
Then we have equivalences for $v=x_n^s$ or $v=y_n$ in the table below.
\[
\begin{array}{c|c|c|c|c|c|c|c|c|c|c|}
w  & y_n & y_{n-1} & \cdots & y_1 & y_0 + y'_0 & \cdots & y^s_{n-2} & y^s_{n-1} & y^s_n\\
\hline
\mathscr{B}_{x_n^s,w} & ||\zb_n,\zb_n|| & ||\zb_{n-1},\zb_n || & \cdots & ||\zb_1,\zb_n ||& |\zb_1,\zb_n| & \cdots & |\zb_{n-1},\zb_n|& |\zb_n,\zb_n| & C \\
\hline
\mathscr{B}_{y_n,w} &  & ||\zb_{n-1},\za_n|| & \cdots & ||\zb_{1},\za_n||& |\zb_1,\za_n| & \cdots & |\zb_1,\za_{n-1}| & C & C \\
\hline
\end{array}
\]
There are no $\mathscr{B}_{v,w}$ where $v=y_n^s$ nor any where $w=x_n$.  This leaves us with three possibilities for fixed values of $w$.  We find equivalences
\[
\begin{array}{c|c|c|c|c|c|c|c|c|c|c|}
v  & x_n & x_{n-1} & \cdots & x_1 & x_0 & x_0^s & \cdots & x_{n-1}^s & x_n^s \\
\hline
\mathscr{B}_{v,x_n^s} & |\za_n,\zb_n| & |\za_{n-1},\zb_n| & \cdots & || \za_1,\zb_n||^0 & || \za_1,\zb_n||& ||\za_1,\zb_n|| &\cdots  & ||\za_n,\zb_n|| & \\
\hline
\mathscr{B}_{v,y_n} & |\za_n,\za_n| & |\za_{n-1},\za_n| & \cdots & || \za_1,\za_n||^0 & || \za_1,\za_n || & ||\za_1,\za_n|| &\cdots & ||\za_n,\za_n || & ||\zb_n,\zb_n||\\
\hline
\end{array},
\]
and the $\mathscr{B}_{v,w}$ where $w=y_n^s$ are all constant. 
\end{proof}

The upshot of Lemmas \ref{lem:quiverConditionsOnSlice1p} and \ref{lem:quiverConditionsOnSlice2} is the following:

\begin{lemma}\label{lem:upshot}
Let $V, W\in \rep_{Q^*(n)}(\bd^*)$. 
Then $f(\eta(V)) \leq f(\eta(W))$, for all $f\in \mathcal{E}_n$ if and only if $g(V) \leq g(W)$ for all $g\in \mathcal{D}_n$. 
Similarly, $f(\eta(V)) = f(\eta(W))$, for all $f\in \mathcal{E}_n$ if and only if $g(V) = g(W)$ for all $g\in \mathcal{D}_n$. 
\end{lemma}

\begin{proof}
The constant type and image type functions are described in Proposition \ref{prop:charRanks} and Lemma \ref{lem:specSlice} respectively. The remaining functions restricted to $\im (\eta)$ are the elements of $\mathcal{E}_n$, and a direct check shows that these are precisely the functions appearing in Lemmas \ref{lem:quiverConditionsOnSlice1p} and \ref{lem:quiverConditionsOnSlice2}. A direct check also shows that the functions of the form $|\gamma,\delta|, ||\gamma,\delta||$, and $||\gamma,\delta||^0$ appearing in Lemmas \ref{lem:quiverConditionsOnSlice1p} and \ref{lem:quiverConditionsOnSlice2} are precisely those in $\mathcal{D}_n$. The result now follows from these same lemmas. 
\end{proof}

We are now ready to prove the main result of this section, Proposition \ref{prop:eta1Orbit}.

\begin{proof}[Proof of Proposition \ref{prop:eta1Orbit}]
That $\overline{O}_{\eta(V)}\cap S(\bd^*)$ is a reduced and irreducible subvariety of $S(\bd^*)$ is immediate from Theorem \ref{thm:intersectS}, since it is a factor of the reduced and irreducible variety on the right hand side of \eqref{eq:Psiproperty}. 
To see that $\overline{O}_{\eta(V)}\cap S(\bd^*)$ is a subvariety of $\im(\eta)$, observe that if $f$ is a function of image type, then, by Lemma \ref{lem:specSlice}, $f|_{S(\bd^*)}$ attains its minimum value at $\eta(V)$. On the other hand, by Theorem \ref{thm:Porbitranks}, $f|_{\overline{O}_{\eta(V)}}$ is maximal at $\eta(V)$. Hence $f$ is constant on the intersection ${\overline{O}_{\eta(V)}\cap S(\bd^*)}$, satisfying $f(M,N) = f(\eta(V))$ for all $(M,N) \in {\overline{O}_{\eta(V)}\cap S(\bd^*)}$. Thus, $\overline{O}_{\eta(V)}\cap S(\bd^*)$ is a subvariety of $\im(\eta)$ by Lemma \ref{lem:specSlice}.

We now show that $\overline{O}_{\eta(V)}\cap S(\bd^*)\subseteq \overline{O}_{\eta(W)}\cap S(\bd^*)$ if and only if $f(V)\leq f(W)$ for all $f\in \mathcal{D}_n$. 
All constant type functions are constant on $S(\bd^*)$ and thus are constant on $\overline{O}_{\eta(V)}\cap S(\bd^*)$ and $\overline{O}_{\eta(W)}\cap S(\bd^*)$. Furthermore, as explained in the previous paragraph, all image type functions are also constant on these subvarieties.
Thus, by Theorem \ref{thm:Porbitranks}, the inclusion $\overline{O}_{\eta(V)}\cap S(\bd^*)\subseteq \overline{O}_{\eta(W)}\cap S(\bd^*)$ holds if and only if $f(\eta(V))\leq f(\eta(W))$ for all quiver type functions (i.e. those $\mathscr{U}_v, \mathscr{L}_v, \mathscr{B}_{v,w}$ which are neither of constant type nor of image type). 
However, since $\overline{O}_{\eta(V)}\cap S(\bd^*)$ and  $\overline{O}_{\eta(W)}\cap S(\bd^*)$ are subvarieties of $\im(\eta)$, we may restrict these quiver type functions to $\im(\eta)$ to see that $\overline{O}_{\eta(V)}\cap S(\bd^*)\subseteq \overline{O}_{\eta(W)}\cap S(\bd^*)$ holds if and only if $f(\eta(V))\leq f(\eta(W))$ for all $f\in \mathcal{E}_n$, 
which, by Lemma \ref{lem:upshot}, holds if and only if $f(V) \leq f(W)$ for all $f\in \mathcal{D}_n$.

Then the statement of the previous paragraph with the roles of $V, W$ reversed
shows that $\overline{O}_{\eta(V)}\cap S(\bd^*) = \overline{O}_{\eta(W)}\cap S(\bd^*)$ if and only if $f(V)= f(W)$ for all $f\in \mathcal{D}_n$.
\end{proof}

\section{Proof of the main theorem}\label{sect:mainTheorem}
In this section we prove Theorem \ref{thm:mainTheorem}, fixing the following notation throughout.
Let $Q$ be a type $D$ quiver with dimension vector $\bd$, and let $(Q^*, \bd^*)$ and $X_Q$ be as defined in and immediately before Lemma \ref{lem:QtoQn}.
Denote the total dimension of $\bd^*$ by $d^*$, and let $a = d^*-\bd^*(y_0')$ and $b = d^*-\bd^*(y_0)$. 
Let $G = GL(a+b)$, let $K=GL(a) \times GL(b)$ embedded as two blocks along the diagonal of $G$, and let $P:=P(\bd^*)$ be as in Section \ref{sect:blockstructure}. 
Recall the open subscheme $\im\Psi \subset K \backslash G$ from \eqref{eq:Psimap}.
Let $S:=S(\bd^*)\subset K\backslash G$ and $\eta: \rep_{Q^*}(\bd^*)\rightarrow S$ be as in Section \ref{sect:setupNotation}. Finally, we use the notation from Section \ref{sect:bundles} throughout the proof.

\begin{proof}[Proof of Theorem \ref{thm:mainTheorem}]
(i) The map of \eqref{eq:reptoKG} is defined by $\zO^\dagger:=\overline{\eta(\zO*1)\cdot P}$.  
Note that $\zO^\dagger$ is indeed a $P$-orbit closure in $K\backslash G$ since it is an irreducible, $P$-stable, closed subvariety of $K\backslash G$, which has finitely many $P$-orbits. 

To see that \eqref{eq:reptoKG} is injective, we first show that $\zO^\circ \mapsto \eta(\zO^\circ*1)\cdot P$ is an injective map between the associated sets of orbits. Since each orbit closure is determined by its open orbit, this is enough.
Choose points $V_i \in \zO_i^\circ$ ($i=1,2$), and suppose that $\eta(V_1*1)$ and $\eta(V_2*1)$ lie in the same $P$-orbit.
Then, in particular, $O_{\eta(V_1*1)}\cap S(\bd^*) = O_{\eta(V_2*1)}\cap S(\bd^*)$. Thus,
by Proposition \ref{prop:eta1Orbit}, we get that $f(V_1*1)=f(V_2*1)$ for all $f \in\cD_n$, which implies that $V_1$ and $V_2$ lie in the same $GL(\bd)$-orbit by Proposition \ref{prop:quiverranksorbit}.

(ii) We have $\zO^\dagger \cap \im\Psi \simeq (\zO^\dagger \cap S)\times Z$ by Theorem \ref{thm:intersectS}.
Since the orbit closure $\zO^\dagger$ is reduced and irreducible and $\im\Psi$ is open in $K\backslash G$, the left hand side of this isomorphism is reduced and irreducible. Consequently, the factor $\zO^\dagger \cap S$ on the right hand side is reduced and irreducible. 
Take a projection morphism
\begin{equation}\label{eq:pi}
\pi\colon S \to \rep_{Q^*}(\bd^*) \qquad \text{satisfying} \quad \pi\circ \eta = 1_{\rep_{Q^*}(\bd^*)}.
\end{equation}
 This makes $S$ the total space of a vector bundle on $\rep_{Q^*}(\bd^*)$.  Denote the pullback to $S$ of the open subvariety $X_Q \subseteq \rep_{Q^*}(\bd^*)$ of Lemma \ref{lem:QtoQn} by $S^\circ$, which is therefore open in $S$.  Thus $\zO^\dagger \cap S^\circ$ is reduced and irreducible. Furthermore, since $\zO^\dagger$ is a closed subscheme of $K\backslash G$, the intersection $\zO^\dagger \cap S^\circ$ is a closed subscheme of $S^\circ$.
 
Next we want to show that $\zO^\dagger \cap S^\circ = \eta(\zO\cdot GL(\bd^*))$.
First observe that $\eta(\zO\cdot GL(\bd^*))$ is a reduced and irreducible closed subscheme of $S^\circ$. This follows since $\eta$ restricted at source to $X_Q$ and at target to $S^\circ$ is a closed embedding, and $\zO\cdot GL(\bd^*) = \zO *_{GL(\bd)} GL(\bd^*)$ is a reduced and irreducible closed subscheme of $X_Q$.
Thus, it is enough to check that the equality holds at the level of $k$-points (recall $k$ is algebraically closed).

Let $V\in \zO^\circ$, so that $\zO^\dagger$ is the closure of the $P$-orbit through $\eta(V)$. 
Let $(\zO^\dagger)^\circ$ denote the open $P$-orbit in $\zO^\dagger$. Then $(\zO^\dagger)^\circ \cap S^\circ$ is open in $\zO^\dagger\cap S^\circ$. By Proposition \ref{prop:eta1Orbit}, 
the points in $(\zO^\dagger)^\circ\cap S^\circ$ are those $\eta(W)$, $W\in X_Q$, with $f(W)= f(V)$ for all $f\in \mathcal{D}_n$.
On the other hand, the orbit $\zO^\circ \cdot GL(\bd^*)$ is dense in $\zO \cdot GL(\bd^*)$ and consists of those $W\in X_Q$ with $f(V) = f(W)$ for all $f\in \mathcal{D}_n$. Thus, the set of $k$-points $\{\eta(W)\mid f(W) = f(V),\, f\in \mathcal{D}_n\}$ is dense in the set of $k$-points of $\eta(\zO\cdot GL(\bd^*))$. Since $\zO^\dagger \cap S^\circ$ and $\eta(\zO\cdot GL(\bd^*))$ are both reduced and irreducible closed subschemes of $S^\circ$, and they have a common dense subset of $k$-points, the two varieties are equal.

Finally, the fact that $\eta$ is a closed embedding gives the first isomorphism below, while Proposition \ref{prop:fiberbundle}(ii) (where $G$ is set to $GL(\bd^*)$ and $H$ set to $GL(\bd)$) along with Lemma \ref{lem:KXHG} (using the same $\phi$ in the proof of Lemma \ref{lem:trivialBundleForMainTheorem}) gives the second isomorphism of
\begin{equation}
\zO^\dagger \cap S^\circ\simeq \zO \cdot GL(\bd^*) \simeq \zO \times (GL(\bd) \backslash GL(\bd^*)).
\end{equation}
Combining with the first isomorphism in this part of the proof gives \eqref{eq:mainthmiso}.

(iii) Restricting to the open subscheme $\im \Psi$ and torus $T_S \leq T$ of \eqref{eq:TSdef} induces a homomorphism of equivariant $K$-groups
\begin{equation}
K_T(K \backslash G) \to K_{T_S}(\im \Psi).
\end{equation}
Furthermore, the $T_S$-equivariant isomorphism of Theorem \ref{thm:intersectS} induces 
\begin{equation}
K_{T_S}(\im \Psi) \xto{\sim} K_{T_S}(S \times Z).
\end{equation}
The fact that $Z$ is smooth gives a $T_S$-equivariant regular embedding and induced restriction map \cite[\S5.2.5]{CGbook}:
\begin{equation}\label{eq:StoSZ}
S \into S \times Z, \qquad s \mapsto (s,1,1,1),
\qquad K_{T_S}(S \times Z) \to K_{T_S}(S).
\end{equation}
Tracing the class $[\cO_{\zO^\dagger}] \in K_T(K \backslash G)$ through the maps above results in $[\cO_{\zO^\dagger \cap S}] \in K_{T_S}(S)$, where the last step uses the fact that $\zO^\dagger$ is Cohen-Macaulay \cite[Thm.~1.2]{Brion03} (see also \cite[Theorem 4.4.7]{Perrin14}) so that the higher Tor sheaves as in \cite[Rmk.~5.2.7]{CGbook} vanish.

Recall that the maximal tori $T(\bd^*) \leq GL(\bd^*)$ and $T(\bd) \leq GL(\bd)$ consist of matrices which are diagonal in each factor.
We have an isomorphism $(\theta|_{T(\bd^*)})^{-1}\colon T_S \xto{\sim} T(\bd^*)$ where $\theta$ is defined in  \eqref{eq:blockdiag}.
It is straightforward to see that the vector bundle map $\pi$ in \eqref{eq:pi} 
can be chosen to be equivariant with respect to this isomorphism (we omit an explicit formula, since it is not needed below).  The restriction of $\pi$ to the $T_S$-stable open subvariety $S^\circ$ is thus an equivariant vector bundle over $X_Q$, inducing maps on equivariant Grothendieck groups
\begin{equation}
K_{T_S}(S) \onto K_{T_S}(S^\circ) \xto{\sim} K_{T(\bd^*)}(X_Q)
\end{equation} 
sending the class of a sheaf to the class of its restriction along the equivariant section $\eta|_{X_Q}$ (see \cite[Cor.~5.4.21]{CGbook} or \cite[Thm.~1.7]{Thomason88}).
In particular, the class at the end of the previous paragraph $[\cO_{\zO^\dagger \cap S}] \in K_{T_S}(S)$ is sent to $[\cO_{\zO  \cdot GL(\bd^*)}] \in K_{T(\bd^*)}(X_Q)$. 

We obtain homomorphisms $K_{T(\bd^*)}(X_Q) \to K_{T(\bd)}(X_Q)$ by restricting the torus action via the identification $T(\bd) \leq T(\bd^*)$, and $K_{T(\bd)}(X_Q) \to K_{T(\bd)}(\rep_Q(\bd))$ by pulling back classes along the $T(\bd)$-equivariant regular embedding $\rep_Q(\bd) \into X_Q$, (see Lemma \ref{lem:QtoQn}).
Composing these two maps with the sequence of maps from the above paragraphs completes the definition of the homomorphism $f$ in the theorem statement.  The fact that the last two steps send $[\cO_{\zO\cdot GL(\bd^*)}] \mapsto [\cO_{\zO}]$
uses that $\zO\cdot GL(\bd^*) \simeq \zO^\dagger \cap S$ is Cohen-Macaulay as well, since it is the restriction of the Cohen-Macaulay variety $\zO^\dagger$ along the regular embedding \eqref{eq:StoSZ}.  (Algebraically, this amounts to the quotient of a Cohen-Macaulay ring by a regular sequence still being Cohen-Macaulay.) This completes the proof.
\end{proof}

\section{Failure of simultaneous compatible Frobenius splitting}\label{sect:Frobenius}

In \cite{KR}, we showed that each type $A$ representation variety has a Frobenius splitting which compatibly splits all of its quiver loci. 
This Frobenius splitting is induced from a Frobenius splitting of a partial flag variety which compatibly splits all of its Schubert varieties (in fact, all of its projected Richardson varieties \cite{KnutsonLamSpeyer}). 
In this section, we show that the situation for type $D$ representation varieties is different: 
we provide an example of a type $D$ representation variety in which it is impossible to simultaneously compatibly Frobenius split all codimension-$1$ quiver loci, and a corresponding example, via Theorem \ref{thm:mainTheorem}, of a $K\backslash G$ in which it is impossible to compatibly split all $B$-orbit closures. The latter example implies that it is not possible, in general, to compatibly split all $B$-orbit closures in double Grassmannians, since each $K\backslash G$ can be identified with an open subvariety of a double Grassmannian. 
We assume that our base field $k$ is perfect of characteristic $p>0$, and we refer the reader to \cite{BrionKumar} or \cite{SmithZhang} for background on Frobenius splitting.

\begin{example}\label{eg:typeDFsplit}
Consider the bipartite type $D$ representation space $\rep_Q(\bd)$ consisting of representations of the following form:
\begin{equation*}\label{eq:D4}
\vcenter{\hbox{\begin{tikzpicture}[point/.style={shape=circle,fill=black,scale=.5pt,outer sep=3pt},>=latex]
   \node[outer sep=-2pt] (x) at (-1.5,0.75) {$k^1$};
   \node[outer sep=-2pt] (1) at (-1.5,-0.75) {$k^1$};
  \node[outer sep=-2pt] (2) at (0,0) {$k^2$};
   \node[outer sep=-2pt] (3) at (1.5,0) {$k^1$,};
   \node[outer sep=-2pt] (4) at (6,0) {$A = \begin{bmatrix}a_1 & a_2\end{bmatrix}, B = \begin{bmatrix}b_1& b_2\end{bmatrix}, C = \begin{bmatrix}c_1 &c_2\end{bmatrix}$ };
  \path[->]
  	(x) edge node[auto] {$A$} (2) 
  	(1) edge node[below] {$B$} (2) 
	(3) edge node[auto] {$C$} (2);
	   \end{tikzpicture}}}.
\end{equation*}
Identify $k[\rep_Q(\bd)]$ with the polynomial ring $k[a_1,a_2,b_1,b_2,c_1,c_2]$ so that the three codimension $1$ quiver loci have prime defining ideals $\langle a_1b_2-a_2b_1\rangle$, $\langle a_1c_2-a_2c_1\rangle$, $\langle b_1c_2-b_2c_1\rangle$. 
Thus each codimension $1$ quiver locus is a determinantal variety, and so is Frobenius split.

Next observe that if there were a Frobenius splitting $\phi$ of $\rep_Q(\bd)$ which \emph{simultaneously} compatibly split all three codimension $1$ quiver loci, then the ideal 
 \[I = (\langle a_1b_2-a_2b_1\rangle \cap \langle a_1c_2-c_1a_2\rangle) + \langle b_1c_2-c_1b_2\rangle\] 
would be compatibly split by $\phi$, because sums and intersections of compatibly split ideals are compatibly split. 
But $I$ is not compatibly split:  
$I$ is generated by $f = (a_1b_2-a_2b_1)(a_1c_2-c_1a_2)$ and $g = b_1c_2-c_1b_2$,
 thus the only elements in $I$ that are homogeneous of degree $3$ have the form $hg$ where $h$ is a linear form. In particular, $c_1(a_1b_2-b_1a_2)\notin I$. However, 
$c_1^2(a_1b_2-b_1a_2)^2 =  b_1c_1f + (a_1a_2b_1c_1-a_1^2b_2c_1)g \in I$. Thus $I$ is not radical, and so is not compatibly split (because compatibly split ideals are radical).
\end{example}

In the following example, we let $U, A, \Omega, \zO^\dagger$ be as in Theorem \ref{thm:mainTheorem}.

\begin{example}\label{eg:FsplittingGK}
Let $G = GL(8)$ and $K = GL(4)\times GL(4)$. These are the $G$ and $K$ of Theorem \ref{thm:mainTheorem} which are associated to $Q$ and $\bd$ from Example \ref{eg:typeDFsplit}.  We claim that there is no Frobenius splitting of $K\backslash G$ which compatibly splits all $B$-orbit closures. 
Indeed, suppose otherwise and let $\phi$ be such a Frobenius splitting. 
Then $\phi|_U$ is a Frobenius splitting of the open set $U\subseteq K\backslash G$ which compatibly splits all intersections of $B$-orbit closures with $U$.

In particular, each $U\cap \zO^\dagger$ is compatibly split by $\phi|_U$.  Thus, by Theorem \ref{thm:mainTheorem} (ii) there is a Frobenius splitting of $\rep_Q(\bd)\times A$ that compatibly splits each $\Omega\times A$. This then implies that there is a Frobenius splitting of $\rep_Q(\bd)$ which compatibly splits each quiver locus $\Omega$, which is false by Example \ref{eg:typeDFsplit}.
\end{example}

\begin{remark}
Example \ref{eg:FsplittingGK} addresses \cite[Rmk.~4.5]{AP16}; it explains why authors Achinger and Perrin did not find the Frobenius splitting they needed in order to obtain a proof in the spirit of \cite{MR85} of their result \cite[Thm.~1]{AP16} that $B$-orbit closures in certain multiple flag varieties are normal, Cohen-Macaulay, and have rational resolutions.
\end{remark}

\begin{remark}
At the time of writing this paper, we do not know whether or not every type $D$ quiver locus is Frobenius split (see \cite[Prob.~8.17]{zwarasurvey}).
\end{remark}

\bibliographystyle{alpha}
\bibliography{typeD}
\end{document}